\documentclass{article}

\usepackage[english]{babel}

\usepackage{todonotes}

\usepackage{amsmath}
\usepackage{graphicx}
\usepackage[colorlinks=true, allcolors=blue]{hyperref}

\usepackage{amsmath}
\usepackage{bm} 
\usepackage{graphicx}
\usepackage{amssymb}

\usepackage{amsthm} 

\usepackage[a4paper,top=4.5cm,bottom=5cm,left=4cm,right=4cm]{geometry}

\theoremstyle{definition}  
\newtheorem{assumption}{Assumption}  
\newtheorem{definition}{Definition}  
\newtheorem{remark}{Remark}  
\newtheorem{theorem}{Theorem}

\renewcommand{\d}{\ensuremath{\,\mathrm{d}}}

\title{Transition Path Theory For L\'{e}vy-Type Processes: SDE Representation and Statistics}
\author{
    Yuanfei Huang\thanks{Asia Pacific Center for Theoretical Physics, Pohang 37673, Korea, and Department of Mathematics, City University of Hong Kong, Kowloon, Hong Kong SAR. Email: \href{yuanfei.huang@apctp.org}{yuanfei.huang@apctp.org}} \and
    Xiang Zhou\thanks{Corresponding author. Department of Mathematics, City University of Hong Kong, Kowloon, Hong Kong SAR. Email: \href{xiang.zhou@cityu.edu.hk}{xiang.zhou@cityu.edu.hk}}
}
\date{}
\begin{document}
\maketitle

\begin{abstract}
This paper establishes a  Transition Path Theory (TPT) for L\'{e}vy-type processes, addressing a critical gap in the study of the  transition mechanism between meta-stabile states in non-Gaussian stochastic systems. A key contribution is the rigorous  derivation of the stochastic differential equation (SDE) representation for transition path processes, which share the same distributional properties as transition trajectories, along with a proof of its well-posedness. This result provides a solid theoretical foundation for sampling transition trajectories. The paper also investigates the statistical properties of transition trajectories, including their probability distribution, probability current, and rate of occurrence.
\end{abstract}


\section{Introduction}\label{sec1}

Stochastic processes, such as It\^{o} diffusion processes and L\'{e}vy-type processes, are foundational tools in a wide range of disciplines, including science, engineering, economics, and statistics. A key phenomenon in stochastic systems is \emph{metastability}, where noise induces transitions between coexisting stable states. While metastability driven by Gaussian noise has been extensively studied \cite{freidlin2012random,vanden2010transition}, non-Gaussian noise-induced metastability plays an equally critical role in many complex systems, yet related research remains limited. These processes govern fundamental phenomena across diverse fields, including climate change \cite{national2013abrupt, zheng2020maximum}, active cellular transport and activity \cite{paneru2021transport, sang2022single, chen2015memoryless}, and gene expression dynamics \cite{pal2013early, wu2018levy}.

A metastable state is generally a locally attracting set, corresponding either to a stable equilibrium point or to the neighborhood of an equilibrium point in a deterministic system. Investigations into metastability aim to uncover the underlying mechanisms and mathematical structures that govern these transitions under stochastic influences, providing a deeper understanding of metastability in complex systems. The study of metastability can be categorized into two primary approaches: the first focuses on the properties of individual transition path, and the second mainly considers the statistics of all the transition trajectories under the ergodic condition.

The first approach concerns systems starting from an initial equilibrium point and arriving at another one at a specified time. The primary objective is to quantify the probability weights of all transition paths connecting these two metastable points. Two main mathematical methodologies have been developed to study such transitions: the stochastic analysis-based Onsager--Machlup (OM) functional and Feynman's path integral formalism.  Onsager and Machlup \cite{OM53, MO53} introduced the Onsager--Machlup functional in 1953 for Gaussian systems with constant diffusion coefficients and linear drift terms. D\"{u}rr \cite{Durr1978} and Ikeda \cite{Ikeda1980} derived its expression for general It\^{o} diffusion processes by analyzing trajectories confined within a tube around a curve. The OM functional reduces to the Freidlin–Wentzell (FW) functional in the vanishing noise limit \cite{freidlin2012random}, with minimum action paths crucial for metastability \cite{lin2019quasi,zhou2010study,wan2010study,ge2012analytical,du2021graph,huang2023most, weinan2004minimum, zhou2008adaptive,huang2019characterization}. Extensions include fractional Brownian motion \cite{moret2002onsager}, McKean–Vlasov dynamics \cite{liu2023onsager}, and jump-diffusion systems \cite{chao2019onsager, huang2025probability}. Feynman’s path integral \cite{feynman1966quantum} sums probability weights over trajectories, weighted by the OM action \cite{Kath1981path, weinan2021applied}. While the Stratonovich interpretation reproduces the OM functional’s form \cite{tang2014summing}, stochastic analysis ensures equivalence of the path integral and OM functional regardless of discretization schemes. In Gaussian systems, the exponential form of transition probabilities simplifies path integrals. For non-Gaussian systems, explicit forms are challenging, and Fourier-based methods are often needed \cite{baule2023exponential, hertz2016path}.

\smallskip
The second approach is based on the \textbf{Transition Path Theory (TPT)}.
In this framework, metastable states are treated as measurable sets, such as neighborhoods of stable invariant sets (e.g., equilibrium points or limit cycles). A transition path is defined as a trajectory that starts in one metastable set and reaches another without revisiting the original set. These paths are of particular interest in the study of chemical reactions and thermally activated processes, where understanding the transitions between metastable states is crucial. TPT was formalized by E and Vanden-Eijnden \cite{vanden2006towards, vanden2010transition, vanden2006transition}, who built on foundational ideas from
the potential theory in probability as well as 
transition path sampling algorithm, transition state theory in theoretic chemistry. At the heart of TPT are the \emph{committor functions}, which characterize the probability of the system reaching one metastable state before another. These functions partition the state space into three distinct regions: two meta-stable state regions and the transition region. This partitioning provides a detailed and quantitative understanding of the mechanisms underlying transitions, enabling insights into the pathways and probabilities governing the dynamics. TPT introduces key observables, such as the \emph{transition current}, which describes the net flux of transitions between metastable states, and the \emph{transition rate}, which quantifies the overall kinetics of the system. These quantities are derived from the ensemble of transition trajectories. In \cite{lu2015reactive}, it was shown that the probability law of these trajectories is equivalent to that of a \emph{transition path process}, which is a strong solution to an auxiliary SDE with a singular drift term. The statistics of the transition path process can be recovered through empirical sampling of the original process. As demonstrated in \cite{gao2023transition}, such a transition path process represents the optimal solution to an optimal control problem that minimizes the Kullback–Leibler (KL) divergence between the potential process and the system under study. Furthermore, \cite{gao2023optimal} discusses a stochastic optimal control formulation for transition path problems in an infinite time horizon for Markov jump processes in Polish spaces. Transition path processes provide a theoretical foundation for sampling transition path trajectories, assuming the committor function is available. Several approaches have been proposed for computing committor functions, including deep learning methods \cite{li2019computing, lin2024deep}, tensor network techniques \cite{chen2023committor}, variational principles \cite{kang2024computing}, and semigroup-based methods \cite{li2022semigroup}. For additional references on transition path sampling, readers are referred to \cite{noe2019boltzmann, dellago2002transition, bolhuis2002transition, bolhuis2021transition, dellago2009transition} and the references therein.

Transition path theory has predominantly concentrated on diffusion processes. However, traditional diffusion models based on Gaussian noise often fail to accurately describe complex stochastic phenomena. In contrast, non-Gaussian processes, such as L\'{e}vy processes, have gained increasing recognition for their ability to capture such complexities \cite{kanazawa2020loopy, barthelemy2008levy, song2018neuronal}. While TPT has been successfully extended to Markovian jump processes with discrete states \cite{metzner2009transition}, its applicability to a broader class of systems remains limited. In particular, a rigorous framework for TPT in the context of L\'{e}vy-type processes in $\mathbb{R}^d$ has yet to be developed. This paper aims to address this gap by extending TPT to encompass L\'{e}vy-type processes, thereby broadening its scope to a wider range of stochastic systems.

\smallskip
\textbf{Content of This Paper.}  
In this paper, we extend the theory of transition paths to stochastic dynamical systems driven by L\'{e}vy jump noise in Euclidean space \(\mathbb{R}^d\). Within the framework of classical Transition Path Theory, we focus on the transitions of a stochastic process \(X\) between two metastable sets \(A\) and \(B\). 

The introduction of jumps in L\'{e}vy-type processes presents both challenges and convenience for the study of TPT. On one hand, the generator of jump processes is considerably more complex than that of Gaussian systems, rendering the computation of key quantities (such as the generator of transition path process, transition path distribution, transition current, and transition rate) and the derivation of analytical results significantly more challenging compared to classical TPT. In this paper, we address these challenges by employing rigorous stochastic analysis tailored to non-Gaussian processes. On the other hand, the inherent characteristics of jump noise offer certain advantages. For example, jumps can mitigate issues arising from boundary singularities, since the boundary of metastable set is negligible in Euclidean space in terms of the starting point of transition trajectories, which simplifies the proof of the well-posedness of the SDE representation of the transition path process in Section \ref{sec3.1}. However, the study of L\'{e}vy-type jump processes also requires the introduction of new concepts that are totally different from classical TPT to describe transition trajectories. For instance, the notion of the \emph{starting point}, as discussed in Section \ref{sec3.2}, plays a crucial role in characterizing the transition trajectories for such processes.

\smallskip
\textbf{Summary of key findings}
\begin{itemize}
    \item \textbf{SDE representation of transition paths:}  
          We derived the SDE representation for the transition path process using the \textit{Doob $h$-transform} and established its well-posedness. This theoretical result forms the foundation for analyzing and sampling transition trajectories in L\'{e}vy-type systems, particularly addressing the challenges posed by the rarity of such events.
    \item \textbf{Statistical properties of transition trajectories:}  
          We systematically analyzed key statistical properties of transition trajectories, including the probability distribution of trajectories, probability currents, and transition rates. Using committor functions, we derived expressions for the transition rate and  transition current across dividing surfaces. We further demonstrated the utility of the committor function's level sets as natural dividing surfaces for quantifying transition dynamics, offering a geometrically intuitive framework for analyzing metastable transitions. Our findings highlight significant differences between the L\'{e}vy-type TPT framework and its Gaussian counterpart, particularly in:
          \begin{itemize}
              \item The distribution of starting and ending points of trajectories, which are no longer restricted to the boundaries of metastable sets but may lie within their interiors, reflecting the impact of jump noise on metastable transitions.
          \end{itemize}
\end{itemize}

The structure of this paper can be summarized as follows. In Section \ref{sec2}, we present the foundational preliminaries, including key notations, definitions, and assumptions that serve as the basis for subsequent developments. Section \ref{sec3.1} focuses on the generator of the transition path process, derived using the Doob \(h\)-transform, and establishes the corresponding SDE representation, with well-posedness rigorously proven in Theorem \ref{theo:TPPSDE}. In Section \ref{sec3.2}, we investigate the steady-state distributions of transition trajectories, particularly their exit, entrance, and starting point distributions. Section \ref{sec3.3} derives the full probability distribution of transition trajectories, while Section \ref{sec3.4} computes the probability current associated with these trajectories. Expressions for the transition rate are provided in Section \ref{sec3.5}. Finally, Section \ref{sec4} concludes the paper with a summary of the results and a discussion of their implications.


\section{Preliminaries}\label{sec2}
\subsection{Basic notations and assumptions}
Let $(\Omega, \mathcal{F}, \mathbb{P})$ be a probability space equipped with a filtration $\{\mathcal{F}_t, t \geq 0\}$ satisfying the usual conditions. Denote by $B = (B_t, t \geq 0)$ a $d$-dimensional standard Brownian motion and by $N$ an independent Poisson random measure on $(\mathbb{R}^d \setminus \{0\}) \times \mathbb{R}_+$ with associated compensator $\widetilde{N}$ and intensity measure $\nu$, where $\nu$ is assumed to be a L\'{e}vy measure which satisfies
\begin{equation*}
    \int_{\mathbb{R}^d\backslash\{0\}}(|r|^2\wedge1)\nu(\d r)<\infty.
\end{equation*}
Throughout, we assume that $B$ and $N$ are independent of $\mathcal{F}_0$.

We focus on the c\`{a}dl\`{a}g (i.e., right-continuous with left limits) solutions $X_t \in \mathbb{R}^d$ to the stochastic differential equation 
\begin{equation*} 
    \d X_t = b(X_{t-}) \d t + \sigma(X_{t-}) \d B_t + \int_{|r| < 1} F(X_{t-}, r) \widetilde{N}(\d r, \d t) + \int_{|r| \geq 1} F(X_{t-}, r) N(\d r, \d t),
\end{equation*}
where the mappings $b_i : \mathbb{R}^d \to \mathbb{R}$, $\sigma_{ij} : \mathbb{R}^d \to \mathbb{R}$, and $F_i : \mathbb{R}^d \times \mathbb{R}^d \to \mathbb{R}$ are assumed to be measurable for $1 \leq i, j \leq d$. The large jump term (for $|r| \geq 1$) can be conveniently handled using interlacing techniques (see \cite[Section 6.2]{applebaum2009levy}). Consequently, it is natural to begin by omitting this term and focusing on the study of the equation driven by continuous noise interspersed with small jumps. To this end, we consider the following SDE
\begin{equation}\label{eqn:SDE}
    \d X_t = b(X_{t-}) \d t + \sigma(X_{t-}) \d B_t + \int_{|r| < 1} F(X_{t-}, r) \widetilde{N}(\d r, \d t).
\end{equation}
It is worth noting that the main results presented in Section \ref{sec3} are applicable to processes with large jumps.

This L\'{e}vy-type process (or called a jump-diffusion process) is associated with the following generator \cite{applebaum2009levy}:
\begin{equation}\label{generator}
    \begin{aligned}
         (\mathcal{L}f)(x) = &\ \sum_{i=1}^d b_i(x) \frac{\partial f}{\partial x_i}(x) + \frac{1}{2} \sum_{i, j = 1}^d \Sigma_{ij}(x) \frac{\partial^2 f}{\partial x_i \partial x_j}(x) \\
        &\ + \int_{|r| < 1} \left[ f\left(x + F(x, r)\right) - f(x) - \sum_{i=1}^d F_i(x, r) \frac{\partial f}{\partial x_i}(x) \right] \nu(\d r),
    \end{aligned}
\end{equation}
for each function $f \in C^2_0(\mathbb{R}^d)$ and each point $x \in \mathbb{R}^d$, where $\Sigma = \sigma \sigma^\top$ is a symmetric positive semidefinite matrix.

We assume the following.
\begin{assumption}\label{assp:ergocity}
    The coefficients $b$, $\sigma$, and $F$ ensure the ergodicity of the Markov process $X$ and the existence of a unique invariant probability density $\rho(x) > 0$, which satisfies the following adjoint (or stationary L\'{e}vy--Fokker--Planck) equation:
\begin{equation*}
    \begin{aligned}
        \mathcal{L}^*\rho = &\ -\nabla \cdot \left[ \left(b(x) -  \int_{|r|<1} F(x, r) \nu(\d r) \right) \rho(x)\right] + \frac{1}{2} \nabla^2 : \left( \Sigma(x)\rho(x) \right) \\
        &\ - \nabla\cdot \int_{|r| < 1}   \int_0^1F\left(\mathcal{T}^{-1}_{F,\theta,r}(x),r\right)\rho\left(\mathcal{T}^{-1}_{F,\theta,r}(x)\right)|\mathcal{J}_{F,\theta,r}(x)|\d\theta\  \nu(\d r) = 0,
    \end{aligned}
\end{equation*}
where $\nabla^2 : (\Sigma \rho) = \sum_{i, j} \partial_{ij} \left(\Sigma_{ij} \rho\right)$, and  we assume the inverse of the map \(x \mapsto x + \theta F(x,r)\) exists and denote it as \(\mathcal{T}^{-1}_{F,\theta,r}\), \(\mathcal{J}_{F,\theta,r}\) is the Jacobian matrix for the inverse mapping of \(\mathcal{T}_{F,\theta,r}\), and \(|\mathcal{J}_{F,\theta,r}|\) denotes its determinant.
\end{assumption}

See \cite{huang2025probability,huang2024vy} for the derivation of the L\'{e}vy--Fokker--Planck equation associated with \eqref{eqn:SDE}. For completeness, a detailed derivation is provided in Appendix \ref{secA1}. 

\begin{remark}
    To enhance the clarity and accessibility of our results, we now provide sufficient conditions under which Assumption \ref{assp:ergocity} holds. Specifically, if $\sigma$ and $b$ are locally Lipschitz and the function $F(\cdot, r)$ is globally Lipschitz, i.e., for each $R > 0$, there exists a constant $C > 0$ such that  
\begin{equation}\label{C1}
    \|\sigma(x) - \sigma(y)\|^2 + |b(x) - b(y)|^2 \leq C|x - y|^2, \quad |x|, |y| \leq R,
\end{equation}
and there exists a constant $L > 0$ such that  
\begin{equation}\label{C2}
    \int_{|r| < 1} |F(x, r) - F(y, r)|^2 \nu(\d r) \leq L|x - y|^2, \quad \forall x, y \in \mathbb{R}^d,
\end{equation}
and, in addition, there exist constants $K, M > 0$ such that  
\begin{equation}\label{C3}
    \mathrm{tr}\left[\Sigma(x)\right] + 2\left\langle x, b(x) \right\rangle + \int_{|r| < 1} |F(x, r)|^2 \nu(\d r) \leq K|x|^2 + M, \quad \forall x \in \mathbb{R}^d,
\end{equation}
where $\mathrm{tr}$ denotes the trace and $\langle \cdot, \cdot \rangle$ denotes the inner product in $\mathbb{R}^d$, then, under conditions \eqref{C1}, \eqref{C2}, and \eqref{C3}, the SDE \eqref{eqn:SDE} has a unique strong solution process $X_\bullet$ and at least one invariant measure. For further details, see \cite[Theorem 3.1]{albeverio2010existence} or \cite[Theorem 2.1, 2.4]{xi2019jump}. Moreover, if there exist a positive constant $\alpha$, a compact set $C\subset\mathbb{R}^d$, a measurable function $f:\mathbb{R}^d\mapsto [1,\infty)$, and a twice continuously
differentiable function $V:\mathbb{R}^d\mapsto\mathbb{R}_+$ satisfying
\begin{equation}\label{C4}
    \mathcal{L}V(x)\leq -\alpha f(x) + 1_C(x),\quad \forall x\in\mathbb{R}^d.
\end{equation}
Then the process $X_\bullet$ of \eqref{eqn:SDE} has a unique invariant measure \cite[Proposition 6.5]{xi2019jump}.

One example satisfying  these conditions  \eqref{C1}, \eqref{C2}, \eqref{C3}, and \eqref{C4} is the following double-well system:  
\begin{equation}\label{doublewell}
    \begin{aligned}
        \d X_t = (X_{t-} - X_{t-}^3)\d t + \sigma \d B_t + \sigma_\mathrm{L}\int_{|r| < 1} r \widetilde{N}(\d r, \d t),
    \end{aligned}
\end{equation}
where $\sigma$ and $\sigma_\mathrm{L}$ are constants, $B_\cdot$ denotes a standard Brownian motion in $\mathbb{R}$, and $N$ is a Poisson random measure with L\'{e}vy measure $\nu$ and compensator $\widetilde{N}$. 
When  $\sigma$, $\sigma_\mathrm{L}$ are not large and    $\nu$ is isometric,  one can   choose a sufficiently small $\alpha$,  a sufficiently large $C$, and the functions  $V(x) = |x|^2$, $f \equiv 1$  to   readily verity  that conditions \eqref{C1}, \eqref{C2}, \eqref{C3}, and \eqref{C4} are satisfied.

To theoretically ensure that $X_\bullet$ of \eqref{eqn:SDE} has continuous transition probability densities, stronger conditions on the drift term $b$, the diffusion coefficient $\sigma$, and the jump coefficient $F$ are required. For instance, to the best of the authors' knowledge, 
if $b \in C^\infty(\mathbb{R}^d,\mathbb{R}^d)$ with bounded derivatives of any order (while $b$ itself need not be bounded), $F(\cdot, r) = \sigma_\mathrm{L}r$, $\sigma$ and $\sigma_\mathrm{L}$ are constants, and there exists a $\delta > 0$ such that  
$    \underset{\epsilon \to 0}{\lim\sup} \frac{\left(\int_{|r| \leq \epsilon} |r|^3 \nu(\d r)\right)^{1/3}}{\left(\int_{|r| \leq \epsilon} |r|^2 \nu(\d r)\right)^{(1 + \delta)/2}} < \infty,
$
then $X_\bullet$ has $C^\infty$ transition probability densities, as established in \cite[Theorem 3.4]{bally2024upper}. A concrete example satisfying these conditions is given by $b(x) = -x$, where the non-Gaussian noise is modeled by a compound Poisson process with jump sizes having a smooth integrable distribution density. This is the Ornstein--Uhlenbeck process with both Gaussian and non-Gaussian noises.

For the double-well system \eqref{doublewell}, the drift term does not satisfy the bounded derivative condition, and therefore the existence of a density of its invariant distribution with respect to the Lebesgue measure is not directly guaranteed at the theoretical level. Nevertheless, our results can still be applied through alternative approaches. For instance, we truncate the drift term beyond a threshold $M$ by setting it equal to $M$, and then smooth the resulting function. Denote this modified drift by $b_M$. The system associated with $b_M$ satisfies Assumption~1. When $M$ is sufficiently large, the behavior of the system under $b_M$ can serve as a good approximation of that under the original system \eqref{doublewell}.

\end{remark}

\subsection{Transition trajectories}
Let $A \subset \mathbb{R}^d$ and $B \subset \mathbb{R}^d$ be two $\rho$-measurable, disjoint open subsets of $\mathbb{R}^d$, each with a smooth boundary, but not necessarily connected. Consider a path $\{X_t\}_{t \geq 0}$ of the process \eqref{eqn:SDE}, starting from $x_0 \in A$, in a given realization. Since the process is ergodic, the path $\{X_t\}_{t \geq 0}$ transitions between $A$ and $B$ infinitely often.

For Markov jump processes with a finite state space, the notions of the \textbf{last-exit-before-entrance time} and the \textbf{first-entrance-after-exit time} are well-established. These concepts can be adapted to our setting in continuous space $\mathbb{R}^d$ with minor modifications (see \cite{metzner2009transition}). Given a trajectory $\{X_t\}_{t \geq 0}$ starting from $x_0 \in A$, we define a set of last-exit-before-entrance times and first-entrance-after-exit times $\mathbb{T} = \{\tau_{A,n}^-, \tau_{B,n}^+, \tau_{B,n}^-, \tau_{A,n}^+\}_{n \in \mathbb{Z}_+}$ as follows.

\begin{definition}[Exit and entrance times]\label{def:times}
    Given a trajectory $\{X_t\}_{t \geq 0}$ starting from $x_0 \in A$, the last-exit-$A$-before-entrance-$B$ time $\tau_{A,n}^-$ and the first-entrance-$B$-after-exit-$A$ time $\tau_{B,n}^+$ belong to $\mathbb{T}$ if and only if
    \begin{equation*}
        \begin{aligned}
            \lim_{t \uparrow \tau_{A,n}^-} X_t = x_n^A \in \bar{A}, \quad X_{\tau_{B,n}^+} \in \bar{B}, \quad \mbox{and}\quad X_s \notin A \cup B\quad
            \forall s \in (\tau_{A,n}^-, \tau_{B,n}^+).
        \end{aligned}
    \end{equation*}
    Similarly, the last-exit-$B$-before-entrance-$A$ time $\tau_{B,n}^-$ and the first-entrance-$A$-after-exit-$B$ time $\tau_{A,n}^+$ belong to $\mathbb{T}$ if and only if
    \begin{equation*}
        \begin{aligned}
            \lim_{t \uparrow \tau_{B,n}^-} X_t = x_n^B \in \bar{B}, \quad X_{\tau_{A,n}^+} \in \bar{A}, \quad\mbox{and}\quad X_s \notin A \cup B \quad
            \forall s \in (\tau_{B,n}^-, \tau_{A,n}^+).
        \end{aligned}
    \end{equation*}
\end{definition}
Specially, we let $\tau_{A,0}^+=0$.

Using the set of times $\mathbb{T}$, we now define the following.

\begin{definition}[Transition times]
    The sets $\mathcal{R}_{A \to B}$ and $\mathcal{R}_{B \to A}$ of transition times are given by
    \begin{equation*}
        \begin{aligned}
            \mathcal{R}_{A \to B} = \bigcup_{n \in \mathbb{Z}_+} (\tau_{A,n}^-, \tau_{B,n}^+) \subset \mathbb{R}, \\
            \mathcal{R}_{B \to A} = \bigcup_{n \in \mathbb{Z}_+} (\tau_{B,n}^-, \tau_{A,n}^+) \subset \mathbb{R}.
        \end{aligned}
    \end{equation*}
\end{definition}

Next, we define the notion of transition trajectories.

\begin{definition}[Transition trajectories]\label{def:tp}
    For $n \in \mathbb{Z}_+$, the c\`{a}dl\`{a}g process $Y^n : [0, \infty) \to \mathbb{R}^d$ is defined as
    \begin{equation}\label{eqn:reactive}
        Y_t^n = X_{(t + \tau_{A,n}^-) \wedge \tau_{B,n}^+}.
    \end{equation}
    Together with the left limit $x_n^A=X_{\tau_{A,n}^--}$, i.e., $(x_n^A, Y_{\cdot}^n)$, this forms the $n$th $A \to B$ transition trajectory. We call $x_n^A$ the \textbf{exit point}, $Y^n_{\tau_{A,n}^-}$ the \textbf{starting point}, and $Y^n_{\tau^+_{B,n}}$ the \textbf{entrance point} of the $n$-th transition trajectory. Similarly, the c\`{a}dl\`{a}g process $\tilde{Y}^n : [0, \infty) \to \mathbb{R}^d$ is defined as
    \begin{equation*}
        \tilde{Y}_t^n = X_{(t + \tau_{B,n}^-) \wedge \tau_{A,n}^+}.
    \end{equation*}
    Together with the left limit $x_n^B=X_{\tau_{B,n}^--}$, i.e., $(x_n^B, \tilde{Y}_{\cdot}^n)$, this forms the $n$th $B \to A$ transition trajectory. The exit, starting and entrance points are defined in a same way as above.
\end{definition}

In this paper, we study transition trajectories from $A$ to $B$ and from $B$ to $A$. By ergodicity, we know that the cardinality of $\mathbb{T}$ is almost surely (a.s.) infinite. Moreover, the times $\tau_{A,n}^-$, $\tau_{B,n}^+$, $\tau_{B,n}^-$, and $\tau_{A,n}^+$ form an increasing sequence satisfying  
\begin{equation*}
    \tau_{A,n}^- \leq \tau_{B,n}^+ \leq \tau_{B,n}^- \leq \tau_{A,n}^+ \leq \tau_{A,n+1}^-,
\end{equation*}
for all $n \in \mathbb{Z}_+$ (see Figure \ref{fig:TPTinlustration}).  

\begin{figure}
    \centering
    \includegraphics[width=0.7\linewidth]{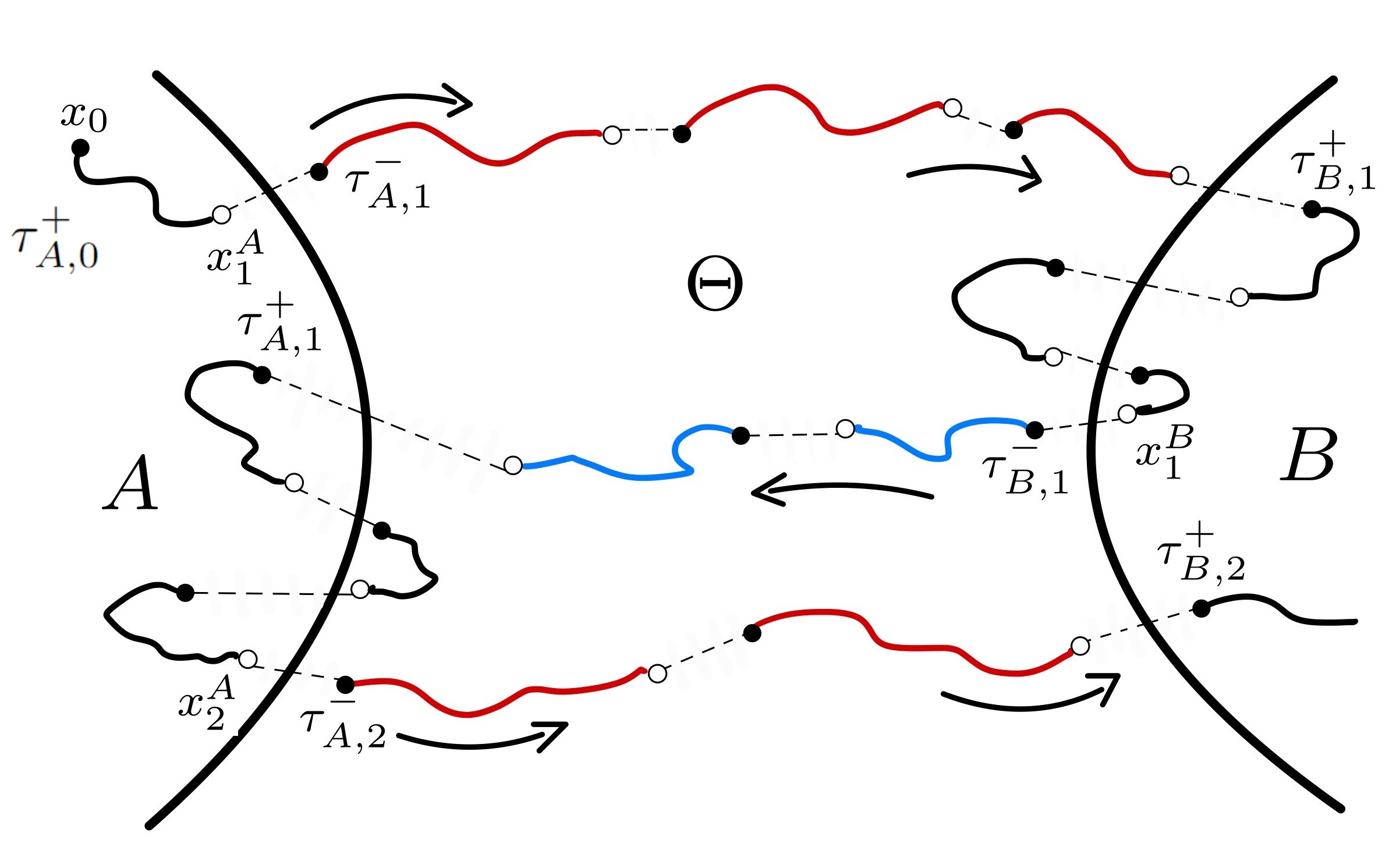}
    \caption{An illustration of a c\`{a}dl\`{a}g trajectory with exit and entrance times. Transition trajectories from $A$ to $B$ are highlighted in red, while the paths from $B$ to $A$ are highlighted in blue (color figure available online). Hollow circles and solid circles connected by dashed lines represent particle jumps, where the particle resides at the position of the solid circle after the jump.}
    \label{fig:TPTinlustration}
\end{figure}

Note, however, that it is possible for $\tau_{A,n}^- = \tau_{B,n}^+$ for some $n \in \mathbb{Z}_+$, corresponding to events where the trajectory jumps directly from $A$ to $B$. Such cases occur only when there exist $x \in A$ and $y \in B$ such that $|x - y| < 1$. If, on the other hand, $\tau_{A,n}^- < \tau_{B,n}^+$, then the trajectory visits states outside of $A$ and $B$ during its transition from $A$ to $B$. In what follows, we consider the jump-diffusion process $X_\bullet$ defined by \eqref{eqn:SDE}, assuming the following.
\begin{assumption}\label{assp:AB}
    For any $x \in A$ and $y \in B$, $|x - y| > 1$. 
\end{assumption} 

For every \(n \in \mathbb{Z}_+\), the distributions of \(x_n^A\) and \(Y_0^n\) depend on the initial point \(x_0 \in A\). Note that, for all \(x_0 \in A\) and \(n \in \mathbb{Z}_+\), the random variable \(Y_0^n\) is distributed on \(A^c \cap (A + 1)\), where 
\begin{equation*}
     A + 1 = \{x + e \mid x \in A, \ e \in \mathbb{R}^d, \ |e| \leq 1\}.
\end{equation*}  
Here, \( A + 1\) represents the set obtained by adding any vector \(e \in \mathbb{R}^d\) with \(|e| \leq 1\) to every point in \(A\). Similar notations are applied to other Borel sets in \(\mathbb{R}^d\). The reason why the distribution of $Y_0^n$ is supported in $A^c \cap (A + 1)$ is that the large jumps of the solution process to \eqref{eqn:SDE} have been eliminated using interlacing techniques. In other words, every transition trajectory can exit $A$ only within a distance of at most 1. If large jumps were not excluded, the distribution of $Y_0^n$ would instead be supported on the entirety of $A^c$. Similarly, the left limits $x_n^A$ of the transition trajectories are supported in $\bar{A} \cap (A + 1)$.  

Due to the right continuity of the solution process $X_\bullet$ of \eqref{eqn:SDE}, the transition trajectories defined in Definition \ref{def:tp} can be classified into two types: 1. $(x_n^A, Y_{\cdot}^n)$ is of type I if $Y_0^n \notin \partial A$. 2. $(x_n^A, Y_{\cdot}^n)$ is of type II if $Y_0^n \in \partial A$. The distribution of $Y_0^n$ on $A^c \cap (A + 1)$ is absolutely continuous with respect to the Lebesgue measure, while $\partial A$ is a set of Lebesgue measure zero. Therefore, from a statistical perspective, it suffices to consider only transition trajectories of type I, as type II trajectories are negligible in terms of the probability measure.

\section{Transition path theory for L\'{e}vy-type processes}\label{sec3}

For the remainder of this paper, we assume that Assumptions \ref{assp:ergocity} and \ref{assp:AB} hold at all times.

\subsection{Transition path process: SDE representation}\label{sec3.1}

Let $\tau_A^+(t)$ and $\tau_B^+(t)$ denote the first hitting times of $X_t$ to the sets $A$ and $B$, respectively:
\begin{equation}\label{tauAtauB}
    \begin{aligned}
        \tau_A^+(t) = \inf\{t' \geq t \mid X_{t'} \in \bar{A}\}, \\
        \tau_B^+(t) = \inf\{t' \geq t \mid X_{t'} \in \bar{B}\}.
    \end{aligned}
\end{equation}
Let $q(x) \geq 0$ denote the \textbf{forward committor function}:
\begin{equation}\label{eqn:forwardcommittor}
    q(x) = \mathbb{P}(\tau_A^+(t) > \tau_B^+(t) \mid X_t = x) := \mathbb{P}_x(\tau_A^+(t) > \tau_B^+(t)).
\end{equation}
The forward committor \(q(x)\) is the probability that the Markov jump process \(X_\bullet\), starting at position \(x\), will reach \(\bar{B}\) before \(\bar{A}\). Consequently, conditioned on avoiding \(\bar{A}\), \(q(x)\) represents the escape probability that \(X_\bullet\), starting from \(x \in \Theta\), will hit the set \(\bar{B}\). It can be demonstrated that the generator of this conditioned process is identical to \(\mathcal{L}\) on \(\Theta\). See Appendix \ref{secA2}.  Recall that, under Assumption \ref{assp:ergocity},
and the smooth boundary of $B$, 
the global well-posedness of \(X_\bullet\) in \eqref{eqn:SDE} is ensured, as is the well-posedness of the conditioned process.  With these properties, the existence and regularity of the escape probability \(q(x)\) can be established, i.e., \(q(x) \) is at least \(C^2(\Theta)\); see \cite[Theorem 2.1]{qiao2013escape} and \cite{ming1989dirichlet}. 
The committor $q$ satisfies the following.
\begin{equation}\label{eqn:pdeq+}
    \begin{cases}
       \mathcal{L} q(x) = 0, & x \in \Theta, \\
       q(x) = 0, & x \in \bar{A}, \\
       q(x) = 1, & x \in \bar{B}.
    \end{cases}
\end{equation}
By the maximum principle \cite[Theorem 2.3]{biswas2025mixed}, $0<q(x)< 1$ for all $x \in \Theta$.  

Let $\overleftarrow{\mathcal{L}}$ denote the adjoint of $\mathcal{L}$ in $L^2(\mathbb{R}^d, \rho(x) \d x)$, given by 
\begin{equation}\label{reversalgenerator}
   \begin{aligned}
       \overleftarrow{\mathcal{L}} u(x) = &\ -b(x) \cdot \nabla u(x) 
       + \frac{1}{\rho(x)}\mathrm{div}(\Sigma(x)\rho(x)) \cdot \nabla u(x) 
       + \frac{1}{2} \mathrm{Tr}(\Sigma \nabla^2 u) \\
       &\ + \int_{|r| < 1} \left[ u\left(\mathcal{T}^{-1}_{F,r}(x)\right) - u(x) - F(x, r) \cdot\nabla  u(x) \right] \frac{\rho\left(\mathcal{T}^{-1}_{F,r}(x)\right)}{\rho(x)} \nu(\d r) \\
       &\ + \int_{|r| < 1} \frac{\rho\left(\mathcal{T}^{-1}_{F,r}(x)\right) + \rho(x)}{\rho(x)} F(x, r) \cdot \nabla u(x) \nu(\d r),
   \end{aligned} 
\end{equation}
where \(\mathcal{T}^{-1}_{F,r}\) is the inverse of the map \(x \mapsto x + F(x,r)\).
This corresponds to the generator of the time-reversed process \(t \mapsto X_{(T-t)-}\) (see also \cite{conforti2022time, privault2004markovian}). However, it is worth noting that \cite{conforti2022time} only addresses processes with pure jump noise, while \cite{privault2004markovian} focuses exclusively on L\'{e}vy processes. For completeness, we provide a simple derivation of the time-reversal generator for general L\'{e}vy-type processes (such as the solution process of \eqref{eqn:SDE}). See Appendix \ref{secA3} for details. The early literature on the derivation of time-reversal generators and related topics for diffusion processes or more general Markovian processes includes important references such as \cite{anderson1982reverse, nagasawa1964time, chung1969reverse,nagasawa1989transformations,nagasawa1979application}. 

In addition to the forward committor function $q(x)$ (recall \eqref{eqn:forwardcommittor}), we define the \textbf{backward committor function} $\overleftarrow{q}(x)$ as the unique solution of
\begin{equation}\label{eqn:pdeq-}
    \begin{cases}
        \overleftarrow{\mathcal{L}} \overleftarrow{q}(x) = 0, & x \in \Theta, \\
        \overleftarrow{q}(x) = 1, & x \in \bar{A}, \\
        \overleftarrow{q}(x) = 0, & x \in \bar{B}.
    \end{cases}
\end{equation}
It is straightforward to verify that
\begin{equation*}
    \overleftarrow{q}(x) = \mathbb{P}_x(\tau^-_B(t) < \tau^-_A(t)),
\end{equation*}
where
\begin{equation*}
    \tau_A^-(t) = \sup\{t' \leq t : X_{t'} \in \bar{A}\}, 
    \quad 
    \tau_B^-(t) = \sup\{t' \leq t : X_{t'} \in \bar{B}\},
\end{equation*}
are the last exit times before time $t$ from $A$ or $B$, respectively.

\begin{definition}[Reversible process]\label{reversal}
    The process \(X_t\) in \eqref{eqn:SDE} is said to be \textbf{reversible} if its generator \(\mathcal{L}\) in \eqref{generator} coincides with the corresponding time reversal operator \(\overleftarrow{\mathcal{L}}\) in \eqref{reversalgenerator}, i.e., \(\mathcal{L} = \overleftarrow{\mathcal{L}}\). Consequently, this implies that \(\overleftarrow{q} = 1 - q\).
\end{definition}
\begin{remark}
This concept originates from thermodynamics, where a reversible process is in equilibrium and the total entropy reaches its maximum value; see, e.g., \cite{huang2025entropy}. In the case of diffusion processes with constant diffusion coefficient, the process \(X_t\) is reversible when its drift term is the negative gradient of some potential function \(U\), i.e., \(b = -\nabla U\). However, for L\'{e}vy-type processes, determining whether the process is reversible is more challenging. The simplest example in the scalar case is a pure jump process with a symmetric L\'{e}vy measure and an initial distribution that is uniform (which, of course, is not a probability distribution in \(\mathbb{R}^d\)).
\end{remark}

For $x \in \Theta$, consider the stopped process $X_{t \wedge \tau_A^+ \wedge \tau_B^+}$ with $X_0 = x$, and let $\mathcal{P}_x$ denote the corresponding measure on $\mathcal{X} = D([0, \infty), \Theta)$:
\begin{equation*}
    \mathcal{P}_x = \mathbb{P}(X \in U \mid X_0 = x), \quad \forall U \in \mathcal{B},
\end{equation*}
where $\mathcal{B}$ is the Borel $\sigma$-algebra on $\mathcal{X}$. If $\Lambda_{AB}$ denotes the event that $\tau_A > \tau_B$, the measure $\mathcal{Q}^q_x$ on $(\mathcal{X}, \mathcal{B})$ defined by
\begin{equation*}
    \frac{\d \mathcal{Q}^q_x}{\d \mathcal{P}_x} = \frac{\mathbf{1}_{\Lambda_{AB}}}{\mathcal{P}_x(\Lambda_{AB})} = \frac{\mathbf{1}_{\Lambda_{AB}}}{q(x)},
\end{equation*}
is absolutely continuous with respect to $\mathcal{P}_x$, provided $x \in \Theta$. By Doob's $h$-transform (see e.g. \cite[Section 4]{chetrite2015nonequilibrium} or \cite[Chapter 11]{chung2005markov}), we know that $\mathcal{Q}^q_x$ defines a Markov process $Y_t$ on $D([0, \infty), A^c)$ with the transition density 
\begin{equation*}
    p^q(t, x, \d y) = \frac{1}{q(x)} p(t, x, \d y) q(y),
\end{equation*}
where $p(t, x, \d y)$ is the transition probability for $X_t$ killed at $\bar{B}$. The corresponding generator of $Y_t$ is 
\begin{equation*}
    \begin{aligned}
        (\mathcal{L}^q f)(x) = &\ \lim_{t \downarrow 0} \frac{\mathbb{E}^q_x(f(X_t)) - f(x)}{t} 
        = q^{-1}(x) \lim_{t \downarrow 0} \frac{\mathbb{E}_x(f(X_t)q(X_t)) - f(x)q(x)}{t} \\ 
        = &\ q^{-1}(x)(\mathcal{L}(fq))(x) \\
        = &\ q^{-1}(x) \left[ \sum_{i=1}^d b_i(x) \left( \frac{\partial (fq)}{\partial x_i} \right)(x) + \frac{1}{2} \sum_{i,j=1}^d \Sigma_{ij}(x) \frac{\partial^2 (fq)}{\partial x_i \partial x_j}(x) \right. \\
        &\ \left. + \int_{|r| < 1} \left[ (fq)(x + F(x, r)) - (fq)(x) - \sum_{i=1}^d F_i(x, r) \frac{\partial (fq)}{\partial x_i}(x) \right] \nu(\d r) \right]\\
         =&\ \sum_{i=1}^d \left( b_i(x) + \sum_{j=1}^d \left(\Sigma_{ij}\frac{\frac{\partial q(x)}{\partial x_j}}{q(x)} \right) -  \int_{|r|<1}F_i(x,r)\nu(\d z) \right)\left( \frac{\partial}{\partial x_i} f \right)(x) \\
         &\  +  \frac{1}{2}\sum_{i,j=1}^d  \left( \Sigma_{ij} \frac{\partial^2}{\partial x_i\partial x_j}f\right)(x) + (\mathcal{L}q)(x)\frac{f(x)}{q(x)} \\
         &\  - \frac{f(x)}{q(x)}\int_{|r|< 1}\left[ q(x+F(x,r)) - q(x) \right]\nu(\d r)\\
        &\  +  q^{-1}(x)\int_{|r|< 1}\left[ (fq)(x+F(x,r)) - (fq)(x) \right]\nu(\d r).
    \end{aligned}
\end{equation*}

Since $\mathcal{L}q = 0$ on $(\bar{A} \cup \bar{B})^c$, we immediately obtain:
\begin{equation*}
    \begin{aligned}
        (\mathcal{L}^q f)(x)
        = &\ \sum_{i=1}^d \left[ b_i(x) + \sum_{j=1}^d \Sigma_{ij}(x) \frac{\partial q(x)/\partial x_j}{q(x)} + \int_{|r| < 1} F_i(x, r) \frac{q(x + F(x, r)) - q(x)}{q(x)} \nu(\d r) \right] \\
        &\ \times \frac{\partial f}{\partial x_i}(x) + \int_{|r| < 1} \left[ f(x + F(x, r)) - f(x) - \sum_{i=1}^d F_i(x, r) \frac{\partial f}{\partial x_i}(x) \right] \\
        &\ \times \frac{q(x + F(x, r))}{q(x)} \nu(\d r) + \frac{1}{2} \sum_{i,j=1}^d \Sigma_{ij}(x) \frac{\partial^2 f}{\partial x_i \partial x_j}(x).
    \end{aligned}
\end{equation*}

The generator $\mathcal{L}^q$ suggests that the $A\to B$ transition trajectories should have the
same law as a solution to the SDE (see e.g. \cite{kurtz2011equivalence}):
\begin{equation}\label{TPPSDE}
    \begin{aligned}
        \d Y_t =&\  \left( b(Y_{t-}) + \frac{\Sigma(Y_{t-})\nabla q(Y_{t-})}{q(Y_{t-})}  + \int_{|r|<1}F(Y_{t-},r)\frac{q(Y_{t-} + F(Y_{t-},r)) -q(Y_{t-})}{q(Y_{t-})}\nu(\d r) \right)  \d t\\
        &\ + \sigma(Y_{t-})\d \widehat{W}_t  + \int_{[0,\infty)\times\{|r|<1\}}1_{[0,\lambda(Y_{s-},r)]}(v)F(Y_{t-},r)\widetilde{\widehat{N}}(\d v,\d r,\d t),
    \end{aligned}
\end{equation}
originating at a point $Y_0=y_0\in A^c\cap(\bar{A}+1)$ and terminating at a point in $\bar{B}$, where $\widehat{W}$ is a standard Brownian motion  in $\mathbb{R}^d$ defined on some probability space and $\widehat{N}$ is an independent Poisson random measure on $[0,\infty)\times(\mathbb{R}^d\backslash\{0\})\times[0,\infty)$ with associated compensator $\widetilde{\widehat{N}}$ and state-independent intensity measure $m\times\nu\times m$, where $m$ is the Lebesgue measure and $\lambda$ is defined as
\begin{equation}
    \lambda(y,r)=\frac{q(y + F(y,r) )}{q(y)}.
\end{equation}

\begin{remark}
    In fact, if we consider \(\lambda\nu\) as a state-dependent L\'{e}vy measure, then the transition path process corresponding to the generator \(\mathcal{L}^q\) can also be viewed as the solution process of the following SDE
    \begin{equation}\label{TPPSDE2}
    \begin{aligned}
        \d Y_t =&\  \left( b(Y_{t-}) + \frac{\Sigma(Y_{t-})\nabla q(Y_{t-})}{q(Y_{t-})}  + \int_{|r|<1}F(Y_{t-},r)\frac{q(Y_{t-} + F(Y_{t-},r)) -q(Y_{t-})}{q(Y_{t-})}\nu(\d r) \right)  \d t\\
        &\ + \sigma(Y_{t-})\d \widehat{W}_t + \int_{\{|r|<1\}} F(Y_{t-},r)\widetilde{N'}(\d r,\d t),
    \end{aligned}
\end{equation}
where \(\widetilde{N'}\) is a state-dependent Poisson random measure with L\'{e}vy measure $\lambda\nu$. It is worth noting that the solution processes of SDEs \eqref{TPPSDE} and \eqref{TPPSDE2} share the same distribution on the path space. The key distinction between the SDEs \eqref{TPPSDE} and \eqref{TPPSDE2} lies in the nature of their jump noise: the jump noise in \eqref{TPPSDE} is independent of the state \(Y_t\), whereas in \eqref{TPPSDE2}, it is state-dependent. Specifically, in SDE \eqref{TPPSDE}, \(Y_t\) can be constructed using a Brownian motion and a Poisson random measure that are independent of the state on the given probability space. This corresponds to the weak solution sense of the SDE.  In contrast, the Poisson random measure in SDE \eqref{TPPSDE2} must be defined in conjunction with the state \(Y_t\), as it explicitly depends on the current state. This state dependence introduces additional complexity into the construction and analysis of the solution.
\end{remark}

\begin{remark}
We now discuss more general cases. It is straightforward to see that the Markov processes corresponding to the generator \(\mathcal{L}^q\) can be extended to more general cases as follows:

1. Diffusion with no jumps:  
    \begin{equation*} 
        \begin{aligned}
            \d Y_t = &\ \left( b(Y_{t-}) + \frac{\Sigma(Y_{t-})\nabla q(Y_{t-})}{q(Y_{t-})}  
             \right) \d t + \sigma(Y_{t-}) \d \widehat{W}_t .
        \end{aligned}
    \end{equation*}  
This SDE representation has been derived in \cite{lu2015reactive} for Gaussian systems.

2. Jump-diffusion with large jumps :  
    \begin{equation*} 
        \begin{aligned}
            \d Y_t = &\ \left( b(Y_{t-}) + \frac{\Sigma(Y_{t-})\nabla q(Y_{t-})}{q(Y_{t-})}  
            + \int_{|r| < 1} F(Y_{t-}, r) \frac{q(Y_{t-} + F(Y_{t-}, r)) - q(Y_{t-})}{q(Y_{t-})} \nu(\d r) \right) \d t \\
            &\ + \sigma(Y_{t-}) \d \widehat{W}_t 
            + \int_{[0, \infty) \times \{|r| < 1\}} 1_{[0, \lambda(Y_{t-}, r)]} F(Y_{t-}, r) \widetilde{\widehat{N}}(\d v, \d r, \d t) \\
            &\ + \int_{[0, \infty) \times \{|r| \geq 1\}} 1_{[0, \lambda(Y_{t-}, r)]} F(Y_{t-}, r) \widehat{N}(\d v, \d r, \d t),
        \end{aligned}
    \end{equation*}  

3. Pure jump process with small and large jumps:  
    \begin{equation*} 
        \begin{aligned}
            \d Y_t = &\ \left( b(Y_{t-}) 
            + \int_{|r| < 1} F(Y_{t-}, r) \frac{q(Y_{t-} + F(Y_{t-}, r)) - q(Y_{t-})}{q(Y_{t-})} \nu(\d r) \right) \d t \\
            &\ + \int_{[0, \infty) \times \{|r| < 1\}} 1_{[0, \lambda(Y_{t-}, r)]} F(Y_{t-}, r) \widetilde{\widehat{N}}(\d v, \d r, \d t) \\
            &\ + \int_{[0, \infty) \times \{|r| \geq 1\}} 1_{[0, \lambda(Y_{t-}, r)]} F(Y_{t-}, r) \widehat{N}(\d v, \d r, \d t).
        \end{aligned}
    \end{equation*}
    The results of this paper are applicable to all three cases. As expected, the corresponding results for Case 1 are consistent with those presented in \cite{lu2015reactive, vanden2006towards}.
\end{remark}

For convenience, let us define the vector field 
\begin{equation}\label{eqn:K}
    K(y)=  b(y) + \frac{\Sigma(y)\nabla q(y)}{q(y)}  + \int_{|r|<1}F(y,r)\frac{q(y + F(y,r)) -q(y)}{q(y)}\nu(\d r).
\end{equation}

\begin{theorem}[Well-posedness of the SDE to transition path process]\label{theo:TPPSDE}
    Let \((\widehat{W}, \mathcal{F}_t^{\widehat{W}})\) be a standard Brownian motion in \(\mathbb{R}^d\), and let \(\widehat{N}\) be an independent Poisson random measure on \([0, \infty) \times (\mathbb{R}^d \setminus \{0\}) \times [0, \infty)\) with intensity measure \(m \times \nu \times m\). Both are defined on a probability space \((\widehat{\Omega}, \widehat{\mathcal{F}}, \mathbb{Q})\).  Let \(\xi : \widehat{\Omega} \to \bar{A}^c \cap (A + 1)\) be a random variable defined on the same probability space and independent of \(\widehat{W}\) and \(\widehat{N}\). Then, there exists a unique c\`{a}dl\`{a}g process \(Y_t : [0, \infty) \to \mathbb{R}^d\), adapted to the augmented filtration \(\widehat{\mathcal{F}}_t\), which satisfies the following equation \(\mathbb{Q}\)-almost surely:
    \begin{equation*}
        \begin{aligned}
            Y_t = &\ \xi + \int_0^{t \wedge \tau_B} K(Y_{s-}) \d s 
            + \int_0^{t \wedge (\tau_A \wedge \tau_B)} \sigma(Y_{s-}) \d \widehat{W}_s \\
            &\ + \int_0^{t \wedge (\tau_A \wedge \tau_B)} 
            \int_{[0, \infty) \times \{|r| < 1\}} 1_{[0, \lambda(Y_s, r)]} F(Y_{s-}, r) \widetilde{\widehat{N}}(\d v, \d r, \d s), \quad t \geq 0,
        \end{aligned}
    \end{equation*}
    where \(\tau_B = \inf\{t \geq 0 \mid Y_t \in \bar{B}\}\) and \(\tau_A = \inf\{t \geq 0 \mid Y_t \in \bar{A}\}\). Moreover, \(Y_t \notin \bar{A}\) for all \(t > 0\).
\end{theorem}
\begin{remark}
    The initial random variable \(\xi\) is assumed to be distributed on \(A^c \cap (A+1)\), as the starting point \(Y_0^n\) of the \(n\)-th transition trajectory is distributed within \(A^c \cap (A+1)\). However, as mentioned in the previous subsection, the ensemble of transition trajectories with starting points at \(\partial A\) is negligible. For simplicity, we only consider \(\xi\) on \(\bar{A}^c \cap (A+1)\). 
\end{remark}
\begin{proof}[Proof of Theorem \ref{theo:TPPSDE}]
When $y_0\in\bar{A}^c\cap(  A+1)$, the existence of a unique strong solution of \eqref{TPPSDE} up to time $\tau_A\wedge\tau_B$ follows from the standard argument since $K(y)$ is Lipschitz continuous in the interior of $\bar{A}^c$. That is, if $y_0\in \bar{A}^c\cap(  A+1)$, there is a unique, c\`{a}dl\`{a}g $\widehat{\mathcal{F}}_t$-adapted process $Y_t$ which satisfies
\begin{equation*}
    \begin{aligned}
        Y_t =&\  y_0 + \int_0^{t\wedge(\tau_A\wedge\tau_B)} K(Y_{s-})\d s + \int_0^{t\wedge(\tau_A\wedge\tau_B)} \sigma(Y_{s-})\d \widehat{W}_s \\
        &\ + \int_0^{t\wedge(\tau_A\wedge\tau_B)} \int_{[0,\infty)\times\{|r|<1\}}1_{[0,\lambda(Y_{s-},r)]}F(Y_{s-},r)\widetilde{\widehat{N}}(\d v,\d r,\d s),\quad t\geq0.
    \end{aligned}
\end{equation*}
Moreover, if $y_0\in \bar{A}^c\cap(  A+1)$, then we must have $\lim_{n\to\infty}\tau_A>\tau_B>0$ almost surely. This follows from an argument similar to the proof of \cite[Section 2]{lu2015reactive}. Specially, for $0<\varepsilon<1$ we define $\tau_\varepsilon=\inf\left\{  t\geq0 \mid q(Y_t)\leq \varepsilon \right\}$. Denote $\tau=\tau_\varepsilon\wedge\tau_B$. We consider the process $z_t=1/q(Y_{t-})\in\mathbb{R}$ which satisfies
\begin{equation*}
    \begin{aligned}
        z_{t\wedge\tau} 
          =&\ z_0 - \int_0^{t\wedge\tau } (z_s)^2\nabla q(Y_{s-}) \cdot \sigma(Y_{s-})\d\widehat{W}_s \\
          &\ + \int_0^{t\wedge\tau }\int_{[0,\infty)\times\{|r|<1\}}1_{[0,\lambda(Y_{s-},r)]}\left[ \frac{1}{q(Y_{s-} + F(Y_{s-}.r) )} -z_s \right]\widetilde{\widehat{N}}(\d v,\d r,\d s).
    \end{aligned}
\end{equation*}
To show this, by using the It\^{o}'s formula (see e.g. \cite[Theorem 4.4.7]{applebaum2009levy}), we obtain that
\begin{equation*}
    \begin{aligned}
        &z_{t\wedge\tau } =\  z_0 - \int_0^{t\wedge\tau }\left(\frac{\nabla q(Y_{s-})}{q^2(Y_{s-})}\cdot(K(Y_{s-})\d s + \sigma(Y_{s-})\d\widehat{W}_s) + \int_{|r|<1}\left[ \frac{1}{q(Y_{s-} + F(Y_{s-}.r) )} \right.\right. \\
        &\ \left.\left. -\frac{1}{q(Y_{s-})} + \frac{\nabla q(Y_{s-})}{q^2(Y_{s-})} \cdot F(Y_{s-},r) \right]\frac{q(Y_{s-}+F(Y_{s-},r))}{q(Y_{s-})}\nu(\d r)\right)\d s\\
        &\ -\frac{1}{2}\int_0^{t\wedge\tau} \mathrm{Tr}\left[ \frac{ ( \nabla^2 q(Y_{s-}))q^2(Y_{s-}) - 2 \nabla q(Y_{s-})(\nabla q(Y_{s-}))^\top q(Y_{s-})  }{q^4(Y_{s-})}\Sigma(Y_{s-})\right]\d s\\
        &\ + \int_0^{t\wedge\tau}\int_{[0,\infty)\times\{|r|<1\}}1_{[0,\lambda(Y_{s-},r)]}\left[ \frac{1}{q(Y_{s-} + F(Y_{s-}.r) )} -\frac{1}{q(Y_{s-})} \right]\widetilde{\widehat{N}}(\d v,\d r,\d s).
    \end{aligned}
\end{equation*}
Now it is sufficient to show that
\begin{equation}\label{eqn:extra=0}
    \begin{aligned}
        &\ - \int_0^{t\wedge\tau }\frac{\nabla q(Y_{s-})}{q^2(Y_{s-})}\cdot K(Y_{s-})\d s  + \int_0^{t\wedge\tau}\int_{|r|<1}\left[ \frac{1}{q(Y_{s-} + F(Y_{s-}.r) )} -\frac{1}{q(Y_{s-})} \right. \\
        &\ \left.  + \frac{\nabla q(Y_{s-})}{q^2(Y_{s-})} \cdot F(Y_{s-},r) \right]\frac{q(Y_{s-}+F(Y_{s-},r))}{q(Y_{s-})}\nu(\d r)\d s\\
        &\ -\frac{1}{2}\int_0^{t\wedge\tau} \mathrm{Tr}\left[ \frac{ ( \nabla^2 q(Y_{s-}))q^2(Y_{s-}) - 2 \nabla q(Y_{s-})(\nabla q(Y_{s-}))^\top q(Y_{s-})  }{q^4(Y_{s-})}\Sigma(Y_{s-})\right]\d s=0.
    \end{aligned}
\end{equation}
This can be verified via a direct calculation by substituting the expression of $K$ in \eqref{eqn:K} into \eqref{eqn:extra=0}, and by using the fact that $\mathcal{L}q(x)=0$ for $x\in\Theta$, as follows,
\begin{equation*}
    \begin{aligned}
         &- \int_0^{t\wedge\tau }\frac{\nabla q(Y_{s-})}{q^2(Y_{s-})}\cdot K(Y_{s-})\d s  + \int_0^{t\wedge\tau}\int_{|r|<1}\left[ \frac{1}{q(Y_{s-} + F(Y_{s-}.r) )} \right. \\
        &\ \left. -\frac{1}{q(Y_{s-})} + \frac{\nabla q(Y_{s-})}{q^2(Y_{s-})} \cdot F(Y_{s-},r) \right]\frac{q(Y_{s-}+F(Y_{s-},r))}{q(Y_{s-})}\nu(\d r)\d s\\
        &\ -\frac{1}{2}\int_0^{t\wedge\tau} \mathrm{Tr}\left[ \frac{ ( \nabla^2 q(Y_{s-}))q^2(Y_{s-}) - 2 \nabla q(Y_{s-})(\nabla q(Y_{s-}))^\top q(Y_{s-})  }{q^4(Y_{s-})}\Sigma(Y_{s-})\right]\d s\\
        =&\ - \int_0^{t\wedge\tau }\frac{\nabla q(Y_{s-})}{q^2(Y_{s-})}\cdot\left(  b(Y_{s-}) + \frac{\Sigma(Y_{s-})\nabla q(Y_{s-})}{q(Y_{s-})}  + \int_{|r|<1}F(Y_{s-},r) \right.\\
        &\  \left. \times \frac{q(Y_{s-} + F(Y_{s-},r)) -q(Y_{s-})}{q(Y_{s-})}\nu(\d r) \right)\d s  + \int_0^{t\wedge\tau}\int_{|r|<1}\left[ \frac{1}{q(Y_{s-} + F(Y_{s-}.r) )}  \right. \\
        &\ \left.  -\frac{1}{q(Y_{s-})}  + \frac{\nabla q(Y_{s-})}{q^2(Y_{s-})} \cdot F(Y_{s-},r) \right]\frac{q(Y_{s-}+F(Y_{s-},r))}{q(Y_{s-})}\nu(\d r)\d s\\
       &\ -\frac{1}{2}\int_0^{t\wedge\tau} \mathrm{Tr}\left[ \frac{ ( \nabla^2 q(Y_{s-}))q^2(Y_{s-}) - 2 \nabla q(Y_{s-})(\nabla q(Y_{s-}))^\top q(Y_{s-})  }{q^4(Y_{s-})}\Sigma(Y_{s-})\right]\d s\\
        =&\ -\int_0^{t\wedge\tau}\left(\frac{\nabla q(Y_{s-})\cdot b(Y_{s-})  + \nabla q(Y_{s-})\cdot\Sigma(Y_{s-})\nabla q(Y_{s-})/q(Y_{s-})  }{q^2(Y_{s-})}   -\int_{|r|<1} \frac{\nabla q(Y_{s-})}{q^2(Y_{s-})} \right.\\
        &\ \left. \times F(Y_{s-},r)\nu(\d r)\right)\d s + \int_0^{t\wedge\tau}\int_{|r|<1} \frac{q(Y_{s-})-q(Y_{s-}+F(Y_{s-},r))}{q^2(Y_{s-})}\nu(\d r)\d s\\
        &\ -\frac{1}{2}\int_0^{t\wedge\tau} \mathrm{Tr}\left[ \frac{ ( \nabla^2 q(Y_{s-})) - 2 \nabla q(Y_{s-})(\nabla q(Y_{s-}))^\top /q(Y_{s-})  }{q^2(Y_{s-})}\Sigma(Y_{s-})\right]\d s\\
        =&\ -\int_0^{t\wedge\tau}\frac{\mathcal{L}q(Y_{s-})}{q^2(Y_{s-})}\d s=0.
    \end{aligned}
\end{equation*}

Since $\tau<\infty$ with probability one, we have
\begin{equation*}
    z_0=\mathbb{E}(z_{t\wedge\tau}) \geq \frac{1}{\varepsilon}\mathbb{Q}(\tau_\varepsilon<\tau_B) + \mathbb{Q}(\tau_\varepsilon>\tau_B).
\end{equation*}
Hence $\mathbb{Q}(\tau_\varepsilon<\tau_B)\leq z_0\varepsilon $. So, $\mathbb{Q}(\tau_A <\tau_B)\leq \lim_{\varepsilon\to0}\mathbb{Q}(\tau_\varepsilon<\tau_B)=0$. Now we obtain all we need and the proof is complete.

\end{proof}

It is well known (see \cite[Chapter 11]{chung2005markov}) that Doob's $h$-transform corresponds to conditioning a process on its behavior at its lifetime. Specifically, the transition trajectories of \(X_t\) defined by \eqref{eqn:reactive} form an $h$-process induced by conditioning \(X_t\) to hit the set \(B\) before \(A\). Consequently, these transition trajectories have the same distribution as the process \(Y_t\) defined by \eqref{TPPSDE}, provided that \(Y_0\) has the same distribution as \(X_{\tau^-_{A,n}}\) for all \(n \in \mathbb{N}\) in the path space.  

The following theorem, which is a corollary of \cite[Chapter 11]{chung2005markov}, demonstrates that the law of the transition trajectories of $X_t$ in \eqref{eqn:SDE} corresponds to the law of the process \(Y_t\) of \eqref{TPPSDE} with an appropriate initial condition. For this reason, we refer to the process \(Y_t\) as the \textbf{transition path process}.

\begin{theorem}[Transition trajectories and transition path process share the same distribution]\label{theo:samelaw}
    Let \(X_t\) satisfy the L\'{e}vy SDE \eqref{eqn:SDE}, and let \(Y^n\) denote the \(n\)-th \(A \to B\) transition trajectory defined by \eqref{eqn:reactive}. Let \(Y_t\) be the process defined in Theorem \ref{theo:TPPSDE}. Then, for any bounded and continuous functional \(H: C([0, \infty)) \to \mathbb{R}\), we have  
    \begin{equation*}
        \mathbb{E}[H(Y^n)] = \widehat{\mathbb{E}}\left[H(Y) \ \Big| \ Y_0 \sim X_{\tau_{A,n}^-} \right],
    \end{equation*}
     where $\widehat{\mathbb{E}}$ denotes the expectation taken under the measure  $\mathbb{Q}$. 
\end{theorem}

The processes \(X_t\) and \(Y^\mathrm{N}_T\) may be defined on a probability space different from the one on which \(Y_t\) is defined. The notation \(Y_0 \sim X_{\tau_{A,n}^-}\) in Theorem \ref{theo:samelaw} indicates that \(Y_0\) shares the same law as \(X_{\tau_{A,n}^-}\). In other words, for any Borel set \(U \subset \mathbb{R}^d\),  
\begin{equation*}
    \mathbb{Q}(Y_0 \in U) = \mathbb{P}(X_{\tau_{A,n}^-} \in U).
\end{equation*}  

Theorem \ref{theo:samelaw} currently addresses only part of the transition paths, as the determination of all left limits \(\{x_n^A\}_{n=1}^\infty\) remains unresolved. This issue will be explored further in the next subsection.

\subsection{Transition exit, entrance and starting point distributions}\label{sec3.2}


The distribution of the random point \(X_{\tau_{A,n}^-}\) (as well as the left limit point \(x^A_n\)) depends on the initial condition \(X_0\). However, in the context of sampling transition paths, there are natural choices for the distribution of \(Y_{0}\) and \(Y_{0-}\), which are closely related to the ``stationary distributions'' introduced in \cite{lu2015reactive}. Firstly, we discuss the selection of the distribution of $Y_{0-}$. To formally motivate this distribution, let \(h > 0\) and consider the regularized hitting times:  
\begin{equation}
    \begin{aligned}
        \tau_{A,h} &= \inf \{t \geq h \mid X_t \in \bar{A}\},\\
        \tau_{B,h} &= \inf \{t \geq h \mid X_t \in \bar{B}\}.
    \end{aligned}
\end{equation}

Define:  
\begin{equation*}
    q_h(x) = \mathbb{P}(\tau_{A,h} > \tau_{B,h} \mid X_0 = x).
\end{equation*}
This represents the probability that, at some time \(s \in [0, h]\), the path \(X_t\) starting from \(x \in \bar{A} \cap (\partial A + 1)\) becomes a transition path, meaning it does not return to \(A\) before hitting \(B\). Based on this, the quantity  
\begin{equation*}
    \eta_{A,h} = h^{-1}\rho(x)\mathbb{P}(\tau_{A,h} > \tau_{B,h} \mid X_0 = x) = h^{-1}\rho(x)q_h(x),
\end{equation*}
can be interpreted as the rate at which transition paths exit \(A\) when the system starts from the stationary state \cite{lu2015reactive}. Consequently, a natural choice for the initial distribution of \(Y_{0-} \in \bar{A}\) is:  
\begin{equation*}
    \eta_A(x) = \lim_{h \to 0} \eta_{A,h}(x).
\end{equation*}

By the Markov property, we have:  
\begin{equation*}
    q_h(x) = \int_{\mathbb{R}^d} \mathbb{P}(\tau_A > \tau_B \mid X_0 = y)\rho(h, x, y)\d y = \mathbb{E}[q(X_h) \mid X_0 = x],
\end{equation*}
where \(\rho(t, x, \cdot)\) is the density of \(X_t\) given \(X_0 = x\). Therefore, for any \(x \in \bar{A}\),  
\begin{equation*}
    \lim_{h \to 0} h^{-1}q_h(x) = \lim_{h \to 0} h^{-1}\mathbb{E}[q(X_h) \mid X_0 = x] = \mathcal{L}q(x).
\end{equation*}
Since \(\mathcal{L}q = 0\) on \(\Theta\), \(q(x) \equiv 0\) on \(\bar{A}\), and \(q(x) \equiv 1\) on \(\bar{B}\), the distribution \(\mathcal{L}q\) is supported on \((\bar{A} \cup \bar{B}) \cap (\Theta + 1)\).  

Hence,  
\begin{equation*}
    \eta_{A,h}(x) \to \eta_A(x) = \rho(x)\mathcal{L}q(x), \quad x \in \bar{A} \cap (\partial A + 1).
\end{equation*}
This distribution corresponds to the distribution of the starting left limits of the transition paths from \(A\) to \(B\). By switching the roles of \(A\) and \(B\) in the above discussion, it is also natural to define a measure on \(\bar{B} \cap (\partial B + 1)\) as:  
\begin{equation*}
    \eta_{B,h}(x) \to \eta_B(x) = \rho(x)\mathcal{L}q(x), \quad x \in \bar{B} \cap (\partial B + 1).
\end{equation*}

Now we define the following.

\begin{definition}
    The \textbf{$AB$-transition exit distribution} on \(\bar{A} \cap (\partial A + 1)\) and the \textbf{$BA$-transition exit distribution} on \(\bar{B} \cap (\partial B + 1)\) are defined as follows:  
\begin{equation}\label{eqn:etaAB}
    \begin{aligned}
        &\eta^-_A(\d x) = \frac{1}{v^-_A}\rho(x)\mathcal{L}q(x), \quad x \in \bar{A} \cap (\partial A + 1),\\
        &\eta^-_B(\d x) = \frac{1}{v^-_B}\rho(x)\mathcal{L}q(x), \quad x \in \bar{B} \cap (\partial B + 1),
    \end{aligned}
\end{equation}
where \(v^-_A\) and \(v^-_B\) are normalizing constants such that \(\eta^-_A\) and \(\eta^-_B\) define probability measures on \(\bar{A} \cap (\partial A + 1)\) and \(\bar{B} \cap (\partial B + 1)\), respectively.
\end{definition}

In terms of \(\overleftarrow{q}\), we define the following.

\begin{definition}
    The \textbf{$AB$-transition entrance distribution} in \(\bar{B} \cap (B + 1)\) is defined as:  
\begin{equation*}
    \eta^+_B(\d x) = \frac{1}{v_B^+}\rho(x)\overleftarrow{\mathcal{L}}\overleftarrow{q}(x)\d x, \quad x \in \bar{B} \cap (B + 1),
\end{equation*}
and analogously, the \textbf{$BA$-transition entrance distribution} on \(\bar{A} \cap (A + 1)\):  
\begin{equation}\label{etaA+}
    \eta^+_A(\d x) = \frac{1}{v_A^+}\rho(x)\overleftarrow{\mathcal{L}}\overleftarrow{q}(x)\d x,
\end{equation}
where \(v_A^+\) and \(v_B^+\) are normalizing constants such that these define probability measures.
\end{definition}

We are now ready to discuss the distribution of \(Y_0\). This can be deduced directly from the above argument as follows. Since \(Y_0\) represents the position of a jump from \(Y_{0-} \sim \eta_A^-\), we know that the probability density of \(Y_t = x \in A^c \cap (\partial A + 1)\) is given by:  
\begin{equation*}
    \eta_A^{*}(x) = v_A^{*-1} \int_{\bar{A} \cap (\partial A + 1)} \int_{|r| < 1} \mathbf{1}_{\{\mathrm{dist}(y, A^c) \leq F(y, r)\}} \nu(\d r) \eta_A^-(y)\d y,
\end{equation*}
where \(v_A^*\) is a normalization constant ensuring that \(\eta_A^*\) is a probability distribution. We now formalize this concept with the following definition:

\begin{definition}
    The \textbf{\(AB\)-transition starting point distribution} on \(A^c \cap (\partial A + 1)\) and the \textbf{\(BA\)-transition starting point distribution} on \(B^c \cap (\partial B + 1)\) are defined as follows:
    \begin{equation}\label{def:eta*}
        \begin{aligned}
            \eta_A^{*}(x) &= v_A^{*-1} \int_{\bar{A} \cap (\partial A + 1)} \int_{|r| < 1} \mathbf{1}_{\{\mathrm{dist}(y, A^c) \leq F(y, r)\}} \nu(\d r) \eta_A^-(y)\d y, \quad x \in A^c \cap (\partial A + 1), \\
            \eta_B^{*}(x) &= v_B^{*-1} \int_{\bar{B} \cap (\partial B + 1)} \int_{|r| < 1} \mathbf{1}_{\{\mathrm{dist}(y, B^c) \leq F(y, r)\}} \nu(\d r) \eta_B^-(y)\d y, \quad x \in B^c \cap (\partial B + 1),
        \end{aligned}
    \end{equation}
    where \(v_A^*\) and \(v_B^*\) are normalization constants, ensuring that \(\eta_A^*\) and \(\eta_B^*\) define probability measures on \(A^c \cap (\partial A + 1)\) and \(B^c \cap (\partial B + 1)\), respectively.
\end{definition}

\subsection{Probability distribution of transition trajectories}\label{sec3.3}
In this subsection, we introduce the key objects used to quantify the statistical properties of transition trajectories. These concepts are defined in a manner similar to those in \cite{metzner2009transition, vanden2006towards}. The results presented here extend the findings of \cite{metzner2009transition, vanden2006towards} from diffusion processes to L\'{e}vy-type processes. Consequently, most of the proofs are derived from these works, with only minor modifications to account for the differences introduced by L\'{e}vy-type dynamics.

\begin{definition}[Probability density of transition trajectories]
    The \textbf{distribution of transition trajectories} from \(A\) to \(B\), denoted by \(\rho_{\mathcal{R}_{A \to B}} = (\rho_{\mathcal{R}_{A \to B}}(x))_{x \in \Theta}\), is defined such that for any \(x \in \Theta\),  
    \begin{equation}\label{eqn:mR}
        \rho_{\mathcal{R}_{A \to B}}(x) = \lim_{T \to \infty} \frac{1}{T} \int_0^T \mathbf{1}_{x}(X_t) \mathbf{1}_{\mathcal{R}_{A \to B}}(t) \d t,
    \end{equation}
    where \(\mathbf{1}_C(\cdot)\) denotes the characteristic function of the set \(C\).
\end{definition}

The distribution \(\rho_{\mathcal{R}_{A \to B}}(x)\) represents the stationary probability density that the system is in position \(x\) at time \(t\) and is transition at that time. Equivalently, \(\rho_{\mathcal{R}_{A \to B}}(x)\) can also be expressed as:  
\begin{equation}\label{eqn:mR=P}
    \rho_{\mathcal{R}_{A \to B}}(x) \d x = \mathbb{P}^\mathrm{s}(X_t \in \d x \ \&\ t \in \mathcal{R}_{A \to B}),
\end{equation}
where \(\mathbb{P}^\mathrm{s}\) denotes the probability with respect to the ensemble of stationary trajectories.  

To avoid confusion, note that the random objects in \eqref{eqn:mR=P} are \(X_t\) and \(\mathcal{R}_{A \to B}\). The time \(t\) in this expression is fixed, and \(\rho_{\mathcal{R}_{A \to B}}(x)\) does not depend on \(t\) because we are considering stationary transition trajectories.

How can we determine an expression for \(\rho_{\mathcal{R}_{A \to B}}(x)\)? Consider the scenario where the process \(X_t\) is currently in state \(x \in \Theta\). What is the likelihood that \(X_t\) behaves reactively? Conceptually, this can be viewed as the product of two probabilities:  
1. The probability that the process originates from \(A\) rather than \(B\).  
2. The probability that the process will eventually reach \(B\) instead of returning to \(A\).  This reasoning emphasizes the importance of these two components in the analysis of transition trajectories.

We have the following theorem,
\begin{theorem}[Distribution of $AB$-transition trajectories]\label{theo:reactivedistribution}
    The probability distribution of transition trajectories from $A$
 to $B$ defined in \eqref{eqn:mR} is given by 
 \begin{equation}\label{eqn:mR=piqq}
     \rho_{\mathcal{R}_{A\to B}}(x)=\rho(x) q(x)\overleftarrow{q}(x),\quad x\in\mathbb{R}^d.
 \end{equation}
 \end{theorem}
 \begin{proof}[Proof of Theorem \ref{theo:reactivedistribution}]
    Let \(X_x^{AB,+}(t)\) denote the first position in \(\bar{A} \cup \bar{B}\) reached by \(X_s\) for \(s \geq t\), conditioned on \(X_t = x\). Similarly, let \(X_x^{AB,-}(t)\) denote the last (limit) position in \(\bar{A} \cup \bar{B}\) left by \(X_s\) for \(s \leq t\), conditioned on \(X_t = x\). Equivalently, \(X_x^{AB,-}(t)\) can be interpreted as the first position in \(\bar{A} \cup \bar{B}\) reached by the time-reversed process \(\tilde{X}_s\) for \(s \geq -t\).

Using these definitions, \eqref{eqn:mR} can be rewritten as:  
\begin{equation*}
    \rho_{\mathcal{R}_{A \to B}}(x) = \lim_{T \to \infty} \frac{1}{T} \int_0^T \mathbf{1}_{x}(X_t) \mathbf{1}_{\bar{A}}(X_x^{AB,-}(t)) \mathbf{1}_{\bar{B}}(X_x^{AB,+}(t)) \d t.
\end{equation*}
Taking the limit as \(T \to \infty\) and applying ergodicity along with the strong Markov property,  we observe the following: \(\mathbf{1}_{x}(X_t)\) converges to 
 the invariant measure \(\rho(x)\); \(\mathbf{1}_{\bar{A}}(X_x^{AB,-}(t))\) converges to the probability that the first position reached by the time-reversed process \(\tilde{X}_s\) for \(s \geq -t\) lies in \(\bar{A}\), which is \(\tilde{\mathbb{P}}_x(\tau_B^- > \tau_A^-)\); and \(\mathbf{1}_{\bar{B}}(X_x^{AB,+}(t))\) converges to the probability that the first position reached by \(X_s\) for \(s \geq t\), conditioned on \(X_t = x\), lies in \(\bar{B}\), which is  \(\mathbb{P}_x(\tau_B^+ < \tau_A^+)\). Thus we deduce:  
\begin{equation*}
    \rho_{\mathcal{R}_{A \to B}}(x) = \rho(x) \mathbb{P}_x(\tau_B^+ < \tau_A^+) \tilde{\mathbb{P}}_x(\tau_B^- > \tau_A^-),
\end{equation*}
which corresponds to \eqref{eqn:mR=piqq} by the definitions of \(q\) and \(\overleftarrow{q}\).
 \end{proof}

 Notice that $\rho_{\mathcal{R}_{A\to B}}(x)=0$ if $x\in A\cup B$. Notice also that $\rho_{\mathcal{R}_{A\to B}}$ is not a normalized distribution. In fact, from \eqref{eqn:mR=P},
 \begin{equation*}
     Z_{AB}=\int_{\mathbb{R}^d}\rho_{\mathcal{R}_{A\to B}}(x)\d x=\int_{\mathbb{R}^d}\rho(x) q(x)\overleftarrow{q}(x)\d x<1
 \end{equation*}
 is the probability that the trajectory is transition at some given instance $t$ in time, i.e.,
\begin{equation*}
    Z_{AB}=\mathbb{P}(t\in\mathcal{R}_{A\to B}).
\end{equation*}

The distribution
\begin{equation*}
    \rho_{A\to B}(x)=Z_{AB}^{-1}\rho_{\mathcal{R}_{A\to B}}(x)=Z_{AB}^{-1}\rho(x) q(x)\overleftarrow{q}(x)
\end{equation*}
is then the normalized distribution of transition trajectories which gives the probability
of observing the system in a transition trajectory and in position $x$ at time $t$ conditional
on the trajectory being transition at time $t$.

\begin{remark}
    If the Markov process is reversible, then $q=1-\overleftarrow{q}$ and the probability distribution of transition trajectories reduces to
    \begin{equation*}
        \rho_{\mathcal{R}_{A\to B}}(x)=\rho(x)q(x)(1-q(x))\quad(\mbox{reversible process}).
    \end{equation*}
\end{remark}

The probability distribution \(\rho_{A \to B}(x)\) provides information about the proportion of time \(AB\)-transition trajectories spend in any given subset of \(\Theta\). If \(A\) and \(B\) are metastable sets (i.e., sets with small volumes that concentrate most of the probability), there must exist dynamical bottlenecks between these sets where the \(AB\)-transition trajectories are likely to spend most of their time \cite{debnath2019enhanced,kang2024computing}. Thus, Theorem \ref{theo:reactivedistribution} can be used to identify these dynamical bottlenecks or transition state regions. On the other hand, it may be useful to identify the regions that transition trajectories are likely to visit. To help characterize such regions, we introduce several key objects.

To begin, let \(S\) be a piecewise \(C^1\) surface of co-dimension 1 contained in \(\Omega_{AB}\), and let \(d\sigma_S(x)\) denote the surface element (Lebesgue measure) on \(S\). Define the distribution:  
\[
d\nu_S(x) = C_S^{-1} \rho(x) d\sigma_S(x), \quad \text{with} \quad C_S = \int_S \rho(x) d\sigma_S(x),
\]
which is the distribution supported on \(S\) induced by \(\rho(x)\). The normalization constant \(C_S\) ensures that \(\nu_S\) is a probability measure. The distribution \(\nu_S\) describes where the trajectory of \eqref{eqn:SDE} (not just the transition portions) crosses \(S\) when it hits this surface.

The following proposition introduces the corresponding distribution \(\nu_{S,AB}\), which characterizes where the \(AB\)-transition trajectories hit \(S\) when crossing this surface.

\begin{theorem}[Hitting point distribution of AB-transition trajectories]\label{theo:hittingdistribution}
    The distribution of hitting points on S by the $AB$-transition trajectories is
    \begin{equation}\label{nuSAB}
        \d \nu_{S,AB}(x)=Z^{-1}_{S,AB}q(x)\overleftarrow{q}(x)\rho(x)\d\sigma_S(x),
    \end{equation}
    where 
    $$Z_{S,AB}=\int_{S}q(x)\overleftarrow{q}(x)\rho(x)\d\sigma_S(x).$$
\end{theorem}
\begin{proof}[Proof of Theorem \ref{theo:hittingdistribution}]
   Use the identity
    \begin{align*}
        \nu_{S,AB}(C\cap S)=&\lim_{d\downarrow0}\lim_{T\to\infty}\frac{
\int_{0}^{T} \mathbf{1}(X_t \in C \cap S_d) \mathbf{1}(\tau_B^+ < \tau_A^+) \mathbf{1}(\tau_B^- < \tau_A^-) \, dt
}{
\int_{0}^{T} \mathbf{1}(X_t \in S_d) \mathbf{1}(\tau_B^+ < \tau_A^+) \mathbf{1}(\tau_B^- < \tau_A^-) \, dt
}\\
=&\lim_{d\downarrow0}\lim_{T\to\infty}\frac{
\frac{1}{T}\int_{0}^{T} \mathbf{1}(X_t \in C \cap S_d) \mathbf{1}(\tau_B^+ < \tau_A^+) \mathbf{1}(\tau_B^- < \tau_A^-) \, dt
}{\frac{1}{T}
\int_{0}^{T} \mathbf{1}(X_t \in S_d) \mathbf{1}(\tau_B^+ < \tau_A^+) \mathbf{1}(\tau_B^- < \tau_A^-) \, dt
},
    \end{align*}
    where $C \in \Theta$ is any $m$-measurable subset and $S_d$ is the slab of thickness $d$ around $S$: 
\begin{equation*}
    S_d := \{ x : \text{dist}(x, S) < d \}.
\end{equation*} Proceeding as in the proof of Theorem \ref{theo:reactivedistribution}, by taking the limit \(T \to \infty\), \(\mathbf{1}(X_t \in C \cap S_d)\) converges to the stationary probability that the process \(X_\bullet\) is in \(C \cap S_d\). Similarly, \(\mathbf{1}(X_t \in S_d)\) converges to the stationary probability that the process \(X_\bullet\) is in \(S_d\). Furthermore, by the definitions of \(q\) and \(\overleftarrow{q}\), we deduce that
\begin{equation*}
    \begin{aligned}
        \nu_{S, AB}(C \cap S) 
=& \lim_{d \to 0^+} 
\frac{\int_{C \cap S_d} q(x)\overleftarrow{q}(x)\rho(x) \, dx}{\int_{S_d} q(x)\overleftarrow{q}(x)\rho(x) \, dx}\\
=& \frac{\int_{C \cap S} q(x)\overleftarrow{q}(x)\rho(x) \, d\sigma_S(x)}{\int_{S} q(x)\overleftarrow{q}(x)\rho(x) \, d\sigma_S(x)},
    \end{aligned}
\end{equation*}
which completes the proof.
\end{proof}

\subsection{Probability current of transition trajectories}\label{sec3.4}
We now consider the probability current of \(AB\)-transition trajectories. To motivate this concept, let $Q_R(t,\eta_A^*,x)$ be the density of $Y_t$ of \eqref{TPPSDE}, with $Y_0\sim\eta^*_A$, and killed on $\bar{B}$
\begin{equation*}
    Q_R(t,\eta_A^*,x)=\mathbb{Q}\left[Y_t\in\d x, t<t_B\mid Y_0\sim\eta_A^*\right],
\end{equation*}
and $t_B$ is the first hitting time of $Y_t$ to $\bar{B}$. The density $Q_R(t,\eta_A^*,x)$ satisfies
\begin{equation*}
    \frac{\partial }{\partial t}Q_R(t,\eta_A^*,x)=(\mathcal{L}^q)^*Q_R(t,\eta_A^*,x),\ x\in \Theta,
\end{equation*}
where $(\mathcal{L}^q)^*$ is the adjoint operator of $\mathcal{L}^q$:
\begin{equation*}
    \begin{aligned}
        ((\mathcal{L}^q)^* f)(x) = &\ - \sum_{i=1}^d \frac{\partial}{\partial x_i}\left[ \left( b_i(x) + \sum_{j=1}^d \Sigma_{ij}(x) \frac{\partial q(x)/\partial x_j}{q(x)} - \int_{|r| < 1} F_i(x, r)\nu(\d r) \right) f(x)\right] \\
        &\  - \int_{|r| < 1} \int_0^1  \frac{\partial}{\partial x_i}\sum_{i=1}^d \left( F_i\left(\mathcal{T}^{-1}_{F,\theta,r}(x),r\right) \frac{q\left(\mathcal{T}^{-1}_{F,\theta,r}(x) + F(\mathcal{T}^{-1}_{F,\theta,r}(x), r)\right)}{q\left(\mathcal{T}^{-1}_{F,\theta,r}(x)\right)} \right. \\
        &\  \times f\left(\mathcal{T}^{-1}_{F,\theta,r}(x)\right)|J_{F,\theta,r}(x)|\Bigg)\nu(\d r)   + \frac{1}{2} \sum_{i,j=1}^d \frac{\partial^2 }{\partial x_i \partial x_j}(\Sigma_{ij}f) (x).
    \end{aligned}
\end{equation*}
Integrating from $t=0$ to $t=\infty$, recall the definition of $\rho_{\mathcal{R}_{A\to B}}$, we know that it satisfies
\begin{equation}\label{eqn:Lq*=0}
    (\mathcal{L}^q)^*\rho_{\mathcal{R}_{A\to B}}(x)=0,\ x\in\Theta,
\end{equation}In divergence form, the equation \eqref{eqn:Lq*=0} is
\begin{equation}\label{eqn:divcurrent}
    -\nabla\cdot J_{AB}=0
\end{equation}
where the vector field
    \begin{equation*}
    \begin{aligned}
       J_{AB}(x)= &\  \left(  b(x) +  \Sigma(x) \frac{\nabla q(x)}{q(x)} - \int_{|r| < 1} F_i(x, r)\nu(\d r) \right) \rho_{\mathcal{R}_{A\to B}}(x) \\
        &\  + \int_{|r| < 1} \int_0^1   F\left(\mathcal{T}^{-1}_{F,\theta,r}(x),r\right) \frac{q\left(\mathcal{T}^{-1}_{F,\theta,r}(x) + F(\mathcal{T}^{-1}_{F,\theta,r}(x), r)\right)}{q\left(\mathcal{T}^{-1}_{F,\theta,r}(x)\right)} \\
        &\  \times \rho_{\mathcal{R}_{A\to B}}\left(\mathcal{T}^{-1}_{F,\theta,r}(x)\right)|\mathcal{J}_{F,\theta,r}(x)| \d\theta\nu(\d r)   + \frac{1}{2} \mathrm{div}(\Sigma (x) \rho_{\mathcal{R}_{A\to B}} (x))\\
        = &\  \left(  b(x) +  \Sigma(x) \frac{\nabla q(x)}{q(x)} - \int_{|r| < 1} F_i(x, r)\nu(\d r) \right) \rho(x)q(x)\overleftarrow{q}(x) \\
        &\  + \int_{|r| < 1} \int_0^1   F\left(\mathcal{T}^{-1}_{F,\theta,r}(x),r\right)  q\left(\mathcal{T}^{-1}_{F,\theta,r}(x) + F\left(\mathcal{T}^{-1}_{F,\theta,r}(x), r\right)\right)\rho\left(\mathcal{T}^{-1}_{F,\theta,r}(x)\right) \\
        &\  \times \overleftarrow{q}\left(\mathcal{T}^{-1}_{F,\theta,r}(x)\right)|\mathcal{J}_{F,\theta,r}(x)| \d\theta\nu(\d r)   + \frac{1}{2} \mathrm{div}\left(\Sigma (x) \rho(x)q(x)\overleftarrow{q}(x)\right).
    \end{aligned}
\end{equation*}

Thus, we have the following statement.
\begin{theorem}[Probability Current of \(AB\)-transition trajectories]\label{theo:current}
    The vector field \(J_{AB}: \Theta \to \mathbb{R}^d\) is given by:  
    \begin{equation}\label{JAB}
        \begin{aligned}
            J_{AB}(x) =&\  \left(  b(x) +  \Sigma(x) \frac{\nabla q(x)}{q(x)} - \int_{|r| < 1} F_i(x, r)\nu(\d r) \right) \rho(x)q(x)\overleftarrow{q}(x) \\
        &\  + \int_{|r| < 1} \int_0^1   F\left(\mathcal{T}^{-1}_{F,\theta,r}(x),r\right)  q\left(\mathcal{T}^{-1}_{F,\theta,r}(x) + F\left(\mathcal{T}^{-1}_{F,\theta,r}(x), r\right)\right) \rho\left(\mathcal{T}^{-1}_{F,\theta,r}(x)\right)\\
        &\  \times \overleftarrow{q}\left(\mathcal{T}^{-1}_{F,\theta,r}(x)\right)|\mathcal{J}_{F,\theta,r}(x)| \d\theta\nu(\d r) + \frac{1}{2} \mathrm{div}\left(\Sigma (x) \rho(x)q(x)\overleftarrow{q}(x)\right).
        \end{aligned}
    \end{equation}
    This is the probability current of the \(AB\)-transition trajectories.
\end{theorem}

\subsection{Transition rate}\label{sec3.5}
In this subsection, we derive the average number of transitions from \(A\) to \(B\) per unit time, or equivalently, the average number of transition trajectories observed per unit time. More precisely, let \(\mathrm{N}_T \in \mathbb{Z}_+\) be defined as:  
\begin{equation*}
    \mathrm{N}_T = \max_{n} \left\{n \geq 1 \mid \tau_{B,n}^+ \leq T\right\},
\end{equation*}
where \(\mathrm{N}_T\) represents the number of \(AB\)-transition trajectories in the time interval \([0,T]\).

We now provide the following definition:

\begin{definition}[Transition rate]
    The \textbf{$AB$ transition rate} \(k_{AB}\) is defined as:  
    \begin{equation*}
        k_{AB} = \lim_{T \to \infty} \frac{\mathrm{N}_T}{T},
    \end{equation*}
    which quantifies the rate of transitions from \(A\) to \(B\). Additionally, the limits:  
    \begin{equation*}
        T_{AB} := \lim_{\mathrm{N} \to \infty} \frac{1}{\mathrm{N}} \sum_{n=0}^{\mathrm{N}-1} (\tau_{B,n+1}^+ - \tau_{A,n}^+),
    \end{equation*}
    and  
    \begin{equation*}
        T_{BA} := \lim_{\mathrm{N} \to \infty} \frac{1}{\mathrm{N}} \sum_{n=0}^{\mathrm{N}-1} (\tau_{A,n+1}^+ - \tau_{B,n}^+),
    \end{equation*}
    are called the \textbf{expected transition times} from \(A \to B\) and \(B \to A\), respectively.
\end{definition}

We note that:
\begin{equation*}
    \begin{aligned}
        \frac{1}{k_{AB}} &= \lim_{T \to \infty} \frac{T}{\mathrm{N}_T} \\
        &= \lim_{\mathrm{N} \to \infty} \frac{1}{\mathrm{N}} \sum_{n=0}^{\mathrm{N}-1} (\tau_{A,n+1}^+ - \tau_{A,n}^+) \\
        &= \lim_{\mathrm{N} \to \infty} \frac{1}{\mathrm{N}} \sum_{n=0}^{\mathrm{N}-1} (\tau_{B,n+1}^+ - \tau_{A,n}^+) + \lim_{\mathrm{N} \to \infty} \frac{1}{\mathrm{N}} \sum_{n=0}^{\mathrm{N}-1} (\tau_{A,n+1}^+ - \tau_{B,n+1}^+) \\
        &= T_{AB} + T_{BA}.
    \end{aligned}
\end{equation*}

Recall that \(\tau_B^+\) in \eqref{tauAtauB} denotes the first hitting time of the process \(X\) to \(\bar{B}\). Consider the mean first hitting time:  
\begin{equation*}
    u_B(x) = \mathbb{E}[\tau_B^+ \mid X_0 = x],
\end{equation*}
which is well-defined under Assumption \ref{assp:ergocity}. If $u_B$ is $C^2$-smooth on $\bar{B}^c$, then it satisfies the equation \cite{duan2015introduction}:  
\begin{equation}\label{pde:ub}
    \begin{cases}
       \mathcal{L} u_B(x) = -1, & x \in \bar{B}^c, \\
       u_B(x) = 0, & x \in \bar{B}.
    \end{cases}
\end{equation}

By the definition of \(\eta_A^+\) in \eqref{etaA+}, we obtain:  
\begin{equation*}
    \begin{aligned}
        T_{AB} &= \lim_{\mathrm{N} \to \infty} \frac{1}{\mathrm{N}} \sum_{n=0}^{\mathrm{N}-1} (\tau_{B,n+1}^+ - \tau_{A,n}^+) \\
        &= \mathbb{E}[\tau_B^+ \mid X_0 \sim \eta_A^+] \\
        &= \int_{\bar{A} \cap (\partial A + 1)} \eta_A^+(\d x) u_B(x).
    \end{aligned}
\end{equation*}

Similarly, we have:
\begin{equation*}
    \begin{aligned}
        T_{BA} &= \lim_{\mathrm{N} \to \infty} \frac{1}{\mathrm{N}} \sum_{n=0}^{\mathrm{N}-1} (\tau_{A,n+1}^+ - \tau_{B,n+1}^+) \\
        &= \mathbb{E}[\tau_A^+ \mid X_0 \sim \eta_B^+] \\
        &= \int_{\bar{B} \cap (\partial B + 1)} \eta_B^+(\d x) u_A(x),
    \end{aligned}
\end{equation*}
where \(\tau_A^+\) is the first hitting time of the process \(X\) to \(\bar{A}\), and its mean $u_A(x)=\mathbb{E}[\tau_A^+\mid X_0=x]$ is well-defined under Assumption \ref{assp:ergocity}, if it is $C^2$-smooth on $\bar{A}^c$, then it satisfies the following:
\begin{equation}\label{pde:ua}
    \begin{cases}
       \mathcal{L} u_A(x) = -1, & x \in \bar{A}^c, \\
       u_A(x) = 0, & x \in \bar{A}.
    \end{cases}
\end{equation}
Discussions on the existence and uniqueness of solutions to non-local systems similar to \eqref{pde:ub} and \eqref{pde:ua} can be found in \cite{taira2004semigroups,garroni2002second}. For instance, if \(b \in C^\infty(\mathbb{R}^d, \mathbb{R}^d)\), \(\sigma_{ij} \in C^\infty(\mathbb{R}^d)\), $F(\cdot,r)=r$ and there exists a constant \(a_0 > 0\) such that  
\begin{align*}
    \sum_{i,j}^d \Sigma_{ij}(x) \xi_i \xi_j \geq a_0 |\xi|^2, \quad \forall x, \xi \in \mathbb{R}^d,
\end{align*}
and the L\'{e}vy measure \(\nu\) is the distribution kernel of a properly supported, pseudo-differential operator, then \eqref{pde:ub} and \eqref{pde:ua} admit unique solutions in \(C^\infty(\bar{B}^c)\) and \(C^\infty(\bar{A}^c)\), respectively; see \cite[Theorem 10.4]{taira2004semigroups}. 

Thus, we have the following statement.
\begin{theorem}[Transition rate]
    The $AB$ transition rate is given by
    \begin{equation*}
        k_{AB}= v_A^+\left(\int_{\bar{A} \cap (\partial A + 1)}   u_B(x)\rho(x)\overleftarrow{\mathcal{L}}\overleftarrow{q}(x)\d x\right)^{-1} + v_B^+\left(\int_{\bar{B} \cap (\partial B + 1)}  u_A(x) \rho(x)\overleftarrow{\mathcal{L}}\overleftarrow{q}(x)\d x\right)^{-1}.
    \end{equation*}
\end{theorem}

We now derive the transition rate using another approach. Given any piecewise \(C^1\) surface \(S \subset \Theta\) of co-dimension 1,  
\begin{equation}\label{GammaSAB}
    \Gamma_{S, AB} = \int_S \hat{n}_S(x) \cdot J_{AB}(x) \d\sigma_S(x)
\end{equation}
gives the average net flux of \(AB\)-transition trajectories across \(S\). Here, \(\hat{n}_S(x)\) is the unit outward-pointing normal to \(S\), and \(\d\sigma_S(x)\) is the surface element on \(S\).

In the following, we focus on the transition rates of transition trajectories that across a special class of surfaces, namely the \emph{level sets of $q(x)$} which we adopted the definition from \cite{vanden2006towards} as:
\[
S_+(z) = \{x : q(x) = z\}, \quad (z \in (0, 1)).
\]
Since {\color{blue}{$q(\cdot)$ is $C^2(\Theta)$}}, $S_+(z)$ is a $C^1$ surface of co-dimension 1 which partitions $\mathbb{R}^d$ into two sets, one containing $A$ and the other containing $B$. The family $\{S_+(z)\}_{z \in (0,1)}$ defines a co-dimension 1 foliation of $\Theta$, i.e.
\[
S_+(z) \cap S_+(z') = \emptyset \quad \text{if } z \neq z', \quad \bigcup_{z \in (0,1)} S_+(z) = \Theta.
\]
The foliation $\{S_+(z)\}_{z \in [0,1]}$ is also called the Transition coordinate associated with the function $q(\cdot)$ (and, by extension, we sometimes refer to $q(\cdot)$ itself as the Transition coordinate).

\begin{definition}[Average frequency of transition trajectories
crossing surfaces]
    Let $k_{AB}(z)$ be the transition rate of the $AB$ transition trajectories that across the level set $S_+(z)$, i.e., $k_{AB}(z)=\Gamma_{S_+,AB}.$
\end{definition}

Using $S=S_+(z)$ in \eqref{GammaSAB} and the co-area formula,
\begin{equation*}
    \d \sigma_z(x)=|\nabla q(x)|\delta(q(x)-z)\d x,
\end{equation*}
the average frequency $k_{AB}(z)$ can be expressed as
\begin{equation}
    \begin{aligned}
        k_{AB}(z)=& \int_{S_+(z)} \hat{n}_S(x) \cdot   \left[ \left(  b(x) +  \Sigma(x) \frac{\nabla q(x)}{q(x)} - \int_{|r| < 1} F_i(x, r)\nu(\d r) \right) \rho(x)q(x)\overleftarrow{q}(x) \right. \\
        &\  + \int_{|r| < 1} \int_0^1   F\left(\mathcal{T}^{-1}_{F,\theta,r}(x),r\right)  q\left(\mathcal{T}^{-1}_{F,\theta,r}(x) + F\left(\mathcal{T}^{-1}_{F,\theta,r}(x), r\right)\right)\rho\left(\mathcal{T}^{-1}_{F,\theta,r}(x)\right) \\
        &\   \times \overleftarrow{q}\left(\mathcal{T}^{-1}_{F,\theta,r}(x)\right)|\mathcal{J}_{F,\theta,r}(x)| \d\theta\nu(\d r)  + \frac{1}{2} \mathrm{div}\left(\Sigma (x) \rho(x)q(x)\overleftarrow{q}(x)\right) \Bigg] \d\sigma_z(x)\\
         =& \int_{\Theta} \nabla q(x) \cdot \left[ \left(  b(x) +  \Sigma(x) \frac{\nabla q(x)}{q(x)} - \int_{|r| < 1} F_i(x, r)\nu(\d r) \right) \rho(x)q(x)\overleftarrow{q}(x) \right. \\
        &\  + \int_{|r| < 1} \int_0^1   F\left(\mathcal{T}^{-1}_{F,\theta,r}(x),r\right)  q\left(\mathcal{T}^{-1}_{F,\theta,r}(x) + F\left(\mathcal{T}^{-1}_{F,\theta,r}(x), r\right)\right) \rho\left(\mathcal{T}^{-1}_{F,\theta,r}(x)\right)\\
        &\   \times \overleftarrow{q}\left(\mathcal{T}^{-1}_{F,\theta,r}(x)\right)|\mathcal{J}_{F,\theta,r}(x)| \d\theta\nu(\d r)  + \frac{1}{2} \mathrm{div}\left(\Sigma (x) \rho(x)q(x)\overleftarrow{q}(x)\right) \Bigg] \delta(q(x)-z)\d x.
    \end{aligned}
\end{equation}

Note that the divergence of the probability current is 0, as shown in \eqref{eqn:divcurrent}, thus \eqref{GammaSAB} can be expressed as
\begin{equation*}
     \Gamma_{C, AB} = \int_{\partial C} \hat{n}_{\partial C}(x) \cdot J_{AB}(x) \d\sigma_{\partial C}(x)= \int_C\mathrm{div}J_{AB}(x)\d x=0,
\end{equation*}
where $C\subset \Theta$ with smooth boundary. If $S$ is any dividing surface between $A$ and $B$, $\Gamma_{S,AB}$ in  \eqref{GammaSAB} is independent of the particular dividing surface we pick and is the average frequency of $AB$-transition
trajectories cross that surface, which gives the averaged number of transitions from $A$ to $B$ per unit
of time that cross the surface we pick. Hence, for any $z\in(0,1)$, we know that $k_{AB}(z)=k_{AB}$.

\section{Discussion and conclusion}\label{sec4}

In this paper, we have developed a comprehensive framework for Transition Path Theory applied to L\'{e}vy-type processes, extending the classical TPT framework originally designed for Gaussian stochastic systems. By systematically addressing the challenges introduced by the non-local and discontinuous nature of L\'{e}vy noise, we have introduced new tools and insights for understanding metastable transitions in non-Gaussian stochastic systems. Below, we summarize the key findings, discuss their implications and limitations, and propose potential future research directions.

Similar to the classical TPT framework, committor functions remain central to the study of L\'{e}vy-type processes. However, solving for committor functions involves addressing nonlocal partial differential equations \eqref{eqn:pdeq+} and \eqref{eqn:pdeq-}, as well as \eqref{pde:ub} and \eqref{pde:ua} for calculating the transition rate. These equations pose significant numerical challenges due to their nonlocal nature and the complexities introduced by jump noise.  

To improve the practical applicability of this framework, future research could focus on developing efficient numerical algorithms for:
\begin{itemize}
    \item Sampling transition paths in high-dimensional systems.
    \item Computing committor functions and transition currents.
    \item Investigating the influence of jump noise on metastable transitions in more complex and realistic settings.
\end{itemize}

Overall, our findings provide a robust foundation for understanding rare transitions in non-Gaussian stochastic systems and pave the way for further exploration of the intricate interplay between jump noise and metastability. These advancements hold significant potential for applications in diverse fields such as climate modeling, neuroscience, and financial mathematics, where non-Gaussian noise is a critical factor.

\section*{Acknowledgements}
Xiang Zhou acknowledges the support from 
 Hong Kong General Research Funds  ( 11318522,   11308323).

\appendix

\section{Derivation of the L\'{e}vy--Fokker--Planck equation}\label{secA1}
The generator of the solution process to \eqref{eqn:SDE} is defined on $C_0^2(\mathbb{R}^d)$ as follows \cite{applebaum2009levy,duan2015introduction},
{ \begin{equation}\label{app:generator}
    \begin{aligned}
           (\mathcal{L}f)(x) = &\ \sum_{i=1}^d b_i(x) \frac{\partial f}{\partial x_i}(x) + \frac{1}{2} \sum_{i, j = 1}^d \Sigma_{ij}(x) \frac{\partial^2 f}{\partial x_i \partial x_j}(x) \\
        &\ + \int_{|r| < 1} \left[ f(x + F(x, r)) - f(x) - \sum_{i=1}^d F_i(x, r) \frac{\partial f}{\partial x_i}(x) \right] \nu(\d r),
    \end{aligned}
\end{equation}}for each function $f \in C^2_0(\mathbb{R}^d)$ and each point $x \in \mathbb{R}^d$. Taylor's theorem is used to approximate the function $f(x + F(x,z))$ as follows:
{ \small\begin{equation}\label{taylorexpansion:f}
    \begin{aligned}
         f(x+F(x,z))=&\ 
         f(x) + \sum_{i=1}^d F_i(x,z)\int_0^1 (\partial_if)(x+\theta F(x,z))\d \theta .
    \end{aligned}
\end{equation}}

For every $\theta\in[0,1]$ and $z\in\mathbb{R}^d\backslash\{0\}$, define the mapping
\begin{equation*}
    \begin{aligned}
        \mathcal{T}_{F,\theta,z}:\ x\mapsto x + \theta F(x,z),
    \end{aligned}
\end{equation*}
and assume both are invertible
for every $\theta\in[0,1]$ and $z\in\mathbb{R}^d\backslash\{0\}$.

Substitute \eqref{taylorexpansion:f} to \eqref{app:generator} and let $p_t(\cdot)$ be the probability density for $X_t$ associated with \eqref{eqn:SDE}, we then obtain the expression for the adjoint operator of $\mathcal{L}$ as follows, for any $f\in C^2_0(\mathbb{R}^d)$,
\begin{equation*}
    \begin{aligned}
        \int_{\mathbb{R}^d}\mathcal{L}f (x)p_t( x)\d x
        =&\ -\int_{\mathbb{R}^d}\left(\sum_{i=1}^d \partial_i\left[ \left( b_i(x) - \int_{|r|<1} F_i(x,r)\nu(\d r)\right)p_t( x)\right] \right.\\
        &\ \left. -  \frac{1}{2}\sum_{i,j=1}^d \partial_i\partial_j\left(\Sigma_{ij}(x) p_t( x) \right) \right)f(x)\d x \\
        &\ + \int_{\mathbb{R}^d}\int_{|r|<1}\int_0^1\sum_{i=1}^dF_i\left(\mathcal{T}^{-1}_{F,\theta,r}(x),r\right)  (\partial_if)(x)\\
        &\ \times p_t\left( \mathcal{T}^{-1}_{F,\theta,r}(x)\right)|\mathcal{J}_{F,\theta,r}(x)|\d \theta \nu(\d r)\d x  =:\ \int_{\mathbb{R}^d}f (x)\mathcal{L}^*p_t( x)\d x,
    \end{aligned}
\end{equation*}
where we have used the integration by parts and the fact that $f$ vanishes at infinity, and $\mathcal{J}_{F,\theta,r}$, $\mathcal{J}_{G,\theta,r}$ are the Jacobian matrices for the inverse mappings of $\mathcal{T}_{F,\theta,r}$ and $\mathcal{T}_{G,\theta,r}$ respectively, $|\mathcal{J}_{F,\theta,z}|$, $|\mathcal{J}_{G,\theta,z}|$ denote their determinants.

We can express the corresponding L\'{e}vy--Fokker-Planck equation of \eqref{eqn:SDE} for any $x \in \mathbb{R}^d$ and $t > 0$ as follows:
\begin{equation}\label{eqn:LFPEinfinite}
    \begin{aligned}
        \frac{\partial p_t(x)}{\partial t} 
        =&\ -\sum_{i=1}^d\ \partial_i\left[\left(b_i(x) - \int_{|r|<1} F_i(x,r)\nu(\d r)\right)p_t(x) \right]
         + \frac{1}{2} \sum_{i,j=1}^d\frac{\partial^2}{\partial x_i\partial x_j}  \Sigma_{ij}(x)p_t(x) \\
         &\ - \int_{|r|<1}\int_0^1\sum_{i=1}^d\partial_i\left( F_i\left(\mathcal{T}^{-1}_{F,\theta,r}(x),r\right) p_t\left(\mathcal{T}^{-1}_{F,\theta,r}(x)\right)|\mathcal{J}_{F,\theta,r}(x)| \right) \d \theta \nu(\d r).
    \end{aligned}
\end{equation}

\section{Derivation of the partial differential equation for the committor}\label{secA2}
Recall that, for a position \(x \in \mathbb{R}^d\), the forward committor \(q(x)\) is defined as the probability that the Markov jump process \(X_\bullet\), starting at position \(x\), will reach \(\bar{B}\) before \(\bar{A}\). In other words, \(q(x)\) represents the first entrance probability of the process \(\{X_t, t \geq 0, X_0 = x\}\) with respect to the set \(\bar{B}\), while avoiding the set \(\bar{A}\).  

A common approach to dealing with entrance or hitting probabilities with respect to a specific subset is to modify the process such that the subset becomes an absorbing set. Recall that \(\mathcal{L}\) is the infinitesimal generator of the Markov jump process \(X_\bullet\), and let \(\bar{A} \subset \mathbb{R}^d\) be a nonempty subset. Suppose we are interested in the process resulting from declaring the positions in \(\bar{A}\) as absorbing states. The infinitesimal generator \(\mathcal{L}^A\) of this modified process is given by \cite{schilling2016introduction}:  
\begin{equation*}
    \mathcal{L}^A f(x) =
    \begin{cases}
      \mathcal{L}f(x), & x \in \bar{A}^c, \\
      0, & x \in \bar{A}.
    \end{cases}
\end{equation*}

Thus, \(q(x)\) satisfies \cite{ming1989dirichlet,qiao2013escape,duan2015introduction}:  
\begin{equation*}
    \begin{cases}
        \mathcal{L}^A q(x) = 0, & x \in \bar{B}^c, \\
        q(x) = 1, & x \in \bar{B},
    \end{cases}
\end{equation*}
or equivalently:  
\begin{equation*}
    \begin{cases}
        \mathcal{L} q(x) = 0, & x \in \Theta, \\
        q(x) = 0, & x \in \bar{A}, \\
        q(x) = 1, & x \in \bar{B},
    \end{cases}
\end{equation*}
which completes the derivation.

\section{Derivation of the generator for the time reversal process}\label{secA3}
Now we discuss the process $\{X_t\}_{0\leq t\leq T}$ for time $t$ decreasing from $T$ to 0. To make the process right continuous for decreasing time, we consider the left-limit process $\{\widehat{X}_t=\mathbb{E}(X_{t-}|\mathcal{F}_t)\}$ and $\overset{\leftarrow}{X}_t:=\widehat{X}_{T-t}$. Here the hat notation denotes the time flows in backward direction, from $T$ to 0. Since the transition density $p(\cdot,\cdot|\cdot,\cdot)$ is continuous, $\mathbb{E}(\widehat{X}_t|\mathcal{F}_t )=\mathbb{E}(X_t|\mathcal{F}_t)$. Let $\overset{\leftarrow}{\mathcal{G}}^0_t=\widehat{\mathcal{G}}^0_{T-t}=\sigma \{\widehat{X}_u:\ T-t\leq u\leq T \}$ and $\overset{\leftarrow}{\mathcal{G}}^*_t$ denotes the right-continuous completion of the $\sigma $-field $\mathcal{G}^0_t$ so that the family $\{\overset{\leftarrow}{\mathcal{G}}^0_t\}$ is right-continuous. The process $\overset{\leftarrow}{X}_t$ is defined on probability space $(\Omega,\mathcal{F},\{\overset{\leftarrow}{\mathcal{G}}^*_t\}_{0\leq t\leq T},\mathbb{P})$. Firstly, note that
\begin{equation*}
    \begin{aligned}
        \mathbb{P}\left[\widehat{X}_{T-h}\in \d y \mid \widehat{X}_{T}\in\d x\right]
        =&\ \frac{\mathbb{P}\left[\widehat{X}_{T-h}\in \d y , \widehat{X}_{T}\in\d x\right]}{\mathbb{P}\left[ \widehat{X}_{T}\in\d x\right]}\\
        =&\ \frac{\mathbb{P}\left[\widehat{X}_{T}\in\d x\mid \widehat{X}_{T-h}\in \d y \right]\mathbb{P}\left[ \widehat{X}_{T-h}\in \d y \right]}{\mathbb{P}\left[ \widehat{X}_{T}\in\d x\right]}\\
        =&\ \frac{p(x,T\mid y,T-h)p(y,T-h)}{p(x,T)}\d x\d y.
    \end{aligned}
\end{equation*}
The generator of $\overset{\leftarrow}{X}$ is given as follows.
\begin{equation*}
    \begin{aligned}
    \overset{\leftarrow}{\mathcal{L}}f(x)=&  \lim_{h\downarrow 0}h^{-1}\mathbb{E}\left[f( \overset{\leftarrow}{X}_h) - f(x)\mid \overset{\leftarrow}{X}_0=x\right] \\
     =&  \lim_{h\downarrow 0}\frac{\mathbb{E}\left[f( \widehat{X}_{T-h}) \mid \widehat{X}_{T}=x\right] - f(x)}{h}\\
    =& \lim_{h\downarrow 0}\frac{1}{h}\left[\int_{\mathbb{R}^d}f(y)\mathbb{P}\left[\widehat{X}_{T-h}=y \mid \widehat{X}_{T}=x\right]\d y - f(x)\right]\\
     =& \lim_{h\downarrow 0}\frac{1}{h}\left[\int_{\mathbb{R}^d}f(y)\frac{p(x,T\mid y,T-h)p(y,T-h)}{p(x,T)} \d y - f(x)\right]\\
    =& \lim_{h\downarrow 0} \int_{\mathbb{R}^d}f(y)\frac{p(y,T-h)}{p(x,T)}\frac{p(x,T\mid y,T-h)- p(x,T\mid y,T)}{h} \d y \\
    &  + \lim_{h\downarrow 0} \int_{\mathbb{R}^d}f(y)\frac{p(x,T\mid y,T)}{p(x,T)}\frac{p(y,T-h)- p(y,T)}{h} \d y \\
    =& -\int_{\mathbb{R}^d}f(y)\frac{p(y,T)}{p(x,T)}\frac{\partial p(x,T\mid y,t)}{\partial t}\Big|_{t=T}\d y - \int_{\mathbb{R}^d}f(y)\frac{p(x,T\mid y,T)}{p(x,T)}\frac{\partial p(y,t)}{\partial t}\Big|_{t=T}\d y \\
    =& \int_{\mathbb{R}^d}f(y)\frac{p(y,T)}{p(x,T)}\mathcal{L}p(x,T\mid y,t)\big|_{t=T}\d y - \int_{\mathbb{R}^d}f(y)\frac{p(x,T\mid y,T)}{p(x,T)}\mathcal{L}^*p(y,t)|_{t=T}\d y \\
    =& \int_{\mathbb{R}^d}f(y)\frac{p(y,T)}{p(x,T)}\mathcal{L}p(x,T\mid y,t)\big|_{t=T}\d y - \int_{\mathbb{R}^d}\mathcal{L}\left(f(y)\frac{p(x,T\mid y,T)}{p(x,T)}\right)p(y,t)|_{t=T}\d y ,
    \end{aligned}
    \end{equation*}
    which is understood under the sense of generalized functions. For the first term,
\begin{equation}\label{app:LA1}
    \begin{aligned}
        &\ \int_{\mathbb{R}^d}f(y)\frac{p(y,T)}{p(x,T)}\mathcal{L}p(x,T\mid y,t)\big|_{t=T}\d y\\
        =&\  \int_{\mathbb{R}^d}f(y)\frac{p(y,T)}{p(x,T)} \left[\sum_{i=1}^d b_i(y) \frac{\partial }{\partial y_i}p(x,T\mid y,t) + \frac{1}{2} \sum_{i, j = 1}^d \Sigma_{ij}(y) \frac{\partial^2 }{\partial y_i \partial y_j}p(x,T\mid y,t) \right. \\
        &\ \left. + \int_{|r| < 1} \left( p(x,T\mid y + F(y, r)) - p(x,T\mid y,t) - \sum_{i=1}^d F_i(y, r) \frac{\partial}{\partial y_i}p(x,T\mid y,t) \right) \nu(\d r)\right]\Bigg|_{t=T}\d y\\
        =&\  -\int_{\mathbb{R}^d} \left[\sum_{i=1}^d \frac{\partial }{\partial y_i}\left(f(y)\frac{p(y,T)}{p(x,T)} b_i(y)\right)  - \frac{1}{2} \sum_{i, j = 1}^d \frac{\partial^2 }{\partial y_i \partial y_j}\left(f(y)\frac{p(y,T)}{p(x,T)} \Sigma_{ij}(y) \right)\right]p(x,T\mid y,t)|_{t=T}\d y  \\
        &\ + \int_{\mathbb{R}^d}f(y)\frac{p(y,T)}{p(x,T)} \left[ \int_{|r| < 1} \Bigg( p(x,T\mid y + F(y, r),t) - p(x,T\mid y,t) \right.\\
        &\  \left.\left. - \sum_{i=1}^d F_i(y, r) \frac{\partial}{\partial y_i}p(x,T\mid y,t) \right) \nu(\d r)\right]\Bigg|_{t=T}\d y\\
        =&\ - \sum_{i=1}^d \frac{\partial }{\partial x_i}\left(f(x) b_i(x)\right) - \sum_{i=1}^d f(x) b_i(x)\frac{\frac{\partial p(x,T)}{\partial x_i} }{p(x,T)} + \frac{1}{2} \sum_{i, j = 1}^d \frac{\partial^2 }{\partial x_i \partial x_j}\left(f(x) \Sigma_{ij}(x) \right)\\
        &\ + \sum_{i, j = 1}^d \frac{\partial }{\partial x_i }\left(f(x) \Sigma_{ij}(x) \right)\frac{\frac{\partial p(x,T)}{\partial x_j}}{p(x,T)} + \frac{1}{2} \sum_{i, j = 1}^d \frac{\partial^2 p(x,T)}{\partial x_i \partial x_j} \frac{f(x) \Sigma_{ij}(x)}{p(x,T)}\\
        &\  + \int_{|r|<1}\left[f\left(\mathcal{T}^{-1}_{F,r}(x)\right)\frac{p(\mathcal{T}^{-1}_{F,r}(x),T)}{p(x,T)} -f(x) + \sum_{i=1}^d\frac{\frac{\partial}{\partial x_i}\left(f(x)p(x,T)F_i(x,r)\right)}{p(x,T)} \right]\nu(\d r),
    \end{aligned}
\end{equation}
and for the second term we have 
\begin{equation}\label{app:LA2}
    \begin{aligned}
        &\ \int_{\mathbb{R}^d}\mathcal{L}\left(f(y)\frac{p(x,T\mid y,T)}{p(x,T)}\right)p(y,t)|_{t=T}\d y \\
        = &\ \int_{\mathbb{R}^d}\left\{\sum_{i=1}^d b_i(y) \frac{\partial }{\partial y_i}\left(f(y)\frac{p(x,T\mid y,T)}{p(x,T)}\right) + \frac{1}{2} \sum_{i, j = 1}^d \Sigma_{ij}(y) \frac{\partial^2 }{\partial y_i \partial y_j}\left(f(y)\frac{p(x,T\mid y,T)}{p(x,T)}\right) \right. \\
        &\ + \int_{|r| < 1} \left[ f(y + F(y, r)) \frac{p(x,T\mid y + F(y, r),T)}{p(x,T)} -  f(y)\frac{p(x,T\mid y,T)}{p(x,T)} \right. \\
        &\ \left. \left. - \sum_{i=1}^d F_i(y, r) \frac{\partial }{\partial y_i}\left(f(y)\frac{p(x,T\mid y,T)}{p(x,T)}\right) \right] \nu(\d r) \right\}p(y,t)\Big|_{t=T}\d y\\
        =&\ - \sum_{i=1}^d\frac{f(x)}{p(x,T)}\frac{\partial }{\partial x_i}(b_i(x)p(x,T)) + \frac{1}{2}\sum_{i,j=1}^d\frac{f(x)}{p(x,T)}\frac{\partial^2}{\partial x_i\partial x_j}(\Sigma_{ij}(x)p(x,T))\\
        &\ + \int_{|r|<1}\left[ f(x)\frac{p\left(\mathcal{T}^{-1}_{F,r}(x),T\right)}{p(x,T)} - f(x) + \sum_{i=1}^d\frac{f(x)}{p(x,T)}\frac{\partial}{\partial x_i} (F_i(x,r)p(x,T)) \right]\nu(\d r).
    \end{aligned}
\end{equation}
Combining \eqref{app:LA1} and \eqref{app:LA2} we have
\begin{equation*}
    \begin{aligned}
     \overset{\leftarrow}{\mathcal{L}}f(x) 
         =&\ -\sum_{i=1}^d b_i(x) \frac{\partial f(x)}{\partial x_i}  + \sum_{i,j=1}^d\frac{\partial f(x)}{\partial x_i}\Sigma_{ij}(x)\frac{\frac{\partial}{\partial x_j}(p(x,T))}{p(x,T)}\\
         &\ + \frac{1}{2}\sum_{i,j=1}^d \Sigma_{ij}(x)\frac{\partial^2}{\partial x_i\partial x_j} f(x) + \sum_{i,j=1}^d\frac{\partial \Sigma_{ij}(x)}{\partial x_i}\frac{\partial f(x)}{\partial x_j}\\
         &\  + \int_{|r|<1}\left[f\left(\mathcal{T}^{-1}_{F,r}(x)\right) -f(x)  \right]\frac{p\left(\mathcal{T}^{-1}_{F,r}(x),T\right)}{p(x,T)} \nu(\d r) + \int_{|r|<1} \sum_{i=1}^d F_i(x,r)\frac{\partial f(x)}{\partial x_i}\nu(\d r)\\
         =&\ \sum_{i=1}^d \left[ -b_i(x) + \frac{\frac{\partial }{\partial x_i}(\Sigma_{ij}(x)p(x,T))}{p(x,T)}\right]\frac{\partial f(x)}{\partial x_i}  + \frac{1}{2}\sum_{i,j=1}^d \Sigma_{ij}(x)\frac{\partial^2}{\partial x_i\partial x_j} f(x) \\
         &\  + \int_{|r|<1}\left[f\left(\mathcal{T}^{-1}_{F,r}(x)\right) -f(x)  -\sum_{i=1}^dF_i(x,r)\frac{\partial f(x)}{\partial x_i}\right]\frac{p\left(\mathcal{T}^{-1}_{F,r}(x),T\right)}{p(x,T)} \nu(\d r) \\
         &\ +  \left[\sum_{i=1}^d\int_{|r|<1}  F_i(x,r)\frac{p\left(\mathcal{T}^{-1}_{F,r}(x),T\right)+p(x,T)}{p(x,T)} \nu(\d r) \right]\frac{\partial f(x)}{\partial x_i},
    \end{aligned}
\end{equation*}
which leads to the form of the reversal generator, as shown in \eqref{reversalgenerator}, when the process is considered to be in its stationary state, i.e., \(p(\cdot,T) = \rho(\cdot)\). Furthermore, our result is consistent with those reported in \cite{conforti2022time, privault2004markovian}.

\bibliographystyle{plain}
\bibliography{sn-bibliography}

\end{document}